\documentclass{elsarticle}
\usepackage[utf8]{inputenc}

\usepackage{amsmath,amsfonts,amssymb,amsthm}
\usepackage{hyperref}
\usepackage{xcolor}
\usepackage{tikz,tikz-cd}
\usetikzlibrary{matrix,arrows}
\usepackage{enumitem}
\usepackage{mathabx}
\usepackage{multirow}
\usepackage{float}

\newcommand{\id}{\operatorname{id}}

\newcommand{\End}{\operatorname{End}}
\newcommand{\Aut}{\operatorname{Aut}}

\renewcommand{\S}{\mathcal{S}}
\newcommand{\Z}{\mathbb{Z}}

\newcommand{\Ann}{\operatorname{Ann}}
\newcommand{\Soc}{\operatorname{Soc}}
\newcommand{\Fix}{\operatorname{Fix}}
\newcommand{\Stab}{\operatorname{Stab}}
\newcommand{\Triv}{\operatorname{Triv}}
\newcommand{\op}{\operatorname{op}}
\newcommand{\A}{\operatorname{a}}
\newcommand{\Con}{\operatorname{Con}}
\newcommand{\con}{\operatorname{con}}

\makeatletter
\newtheorem{thm}{Theorem}[section]
\newtheorem{lem}[thm]{Lemma}
\newtheorem{cor}[thm]{Corollary}
\newtheorem{pro}[thm]{Proposition}

\newtheorem{defn}[thm]{Definition}

\newtheorem*{convention}{Convention}
\newtheorem{rem}[thm]{Remark}
\newtheorem{exa}[thm]{Example}

\makeatother

\title{Common divisor graphs for skew braces}
\author[1]{Silvia Properzi}
\ead{silvia.properzi@vub.be}
\author[2]{Arne Van Antwerpen}
\ead{arne.vanantwerpen@ugent.be}
\date{}

\address[1]{
Department of Mathematics and Data Science, Vrije Universiteit Brussel\\ Pleinlaan 2, 1050 Brussel, Belgium }
\address[2]{Department of Mathematics: Algebra and Geometry, Universiteit Gent\\ Krijgslaan 281, 9000 Gent, Belgium }

\begin{document}
\begin{abstract}
We introduce two common divisor graphs associated with a finite skew brace, based on its $\lambda$- and $\theta$-orbits. We prove that the number of connected components is at most two and the diameter of a connected component is at most four. Furthermore, we investigate their relationship with isoclinism. Similarly to its group theoretic inspiration, the skew braces with a graph with two disconnected vertices are very restricted and are determined. Finally, we classify all finite skew braces with a graph with one vertex, where four infinite families arise.
\end{abstract}

\maketitle

\section{Introduction}
In 1992 Drinfel'd \cite{MR1183474} proposed the study of set-theoretic solutions of the Yang-Baxter equation. Recall that a set-theoretic solution of the Yang-Baxter equation is a tuple $(X,r)$ where $X$ is a non-empty set and $r:X \times X \longrightarrow X \times X$ is a map such that on $X^3$ it holds that $$(r\times \id_X)(\id_X \times r)(r \times \id_X)=(\id_X \times r)(r \times \id_X)(\id_X \times r).$$ In particular, the class of bijective non-degenerate set-theoretic solutions has received extensive attention \cite{bachiller2018solutions, MR3647970, jespers2019factorizations,lebed2019structure,rump2014brace,smoktunowicz2022passage}, where recently a full subclass of solutions was characterized by combinatorial objects \cite{agore2023settheoretic}.
In 2007, Rump \cite{MR2278047} introduced the notion of a brace, which was extended by Guarnieri and Vendramin \cite{MR3647970} to skew braces in 2017. These algebraic structures both generate and govern bijective set-theoretic solutions, in the sense that every bijective non-degenerate solution has an associated skew brace and vice versa. Recently, links between skew braces and different fields of mathematics, such as differential geometry \cite{rump2014brace} and pre-Lie algebras \cite{smoktunowicz2022passage}, have been shown and they motivate, in addition to algebraic curiosity, their independent study. 
A skew brace
is a triple $(A,+,\circ)$, where 
$(A,+)$ and $(A,\circ)$ are groups and
the compatibility condition
\[
a\circ(b+c)=a\circ b-a+a\circ c
\]
holds for all $a,b,c\in A$. 

An important tool in the study of skew braces is the group homomorphism $\lambda\colon (A,\circ)\to\Aut(A,+)$, $a\mapsto(b\mapsto\lambda_a(b)=-a+a\circ b)$. It is crucial for the results in this paper, as it encodes an action of 
the group $(A,\circ)$ on $(A,+)$ by automorphisms. 
Furthermore, the map $\lambda$ is closely related to the set-theoretic solutions associated with skew braces. Denote the conjugation by $a\in A$ in $(A,+)$ by $\sigma_a$. If $A$ is a skew brace, 
then the map $r_A\colon A\times A\to A\times A$ is a bijective non-degenerate solution of the Yang-Baxter equation \cite{MR3647970} with, 
\[
r_A(a,b)=(\lambda_a(b),\lambda_{\lambda_a(b)}^{-1}(\sigma_{\lambda_a(b)}(b))).
\]
Moreover, every bijective non-degenerate solution can be constructed from a skew brace~\cite{bachiller2018solutions, MR3763907}.

On the one hand, Jacobson radical rings 
are prominent examples of skew braces~\cite{MR2278047}. This suggests using 
ring-theoretical methods to study the skew braces and the Yang-Baxter equation which was explored in \cite{MR3177933,MR3861714,MR3814340}. On the other hand, 
a group $(A,+)$ gives rise to the trivial skew brace $(A,+,+)$ and $(A,+^{op},+)$, with $a+^{op}b=b+a$ for $a,b \in A$. The latter is of particular importance, as then $\lambda_a$ is exactly conjugation by $-a$, which motivates that both group-theoretical tools and notions are applied to skew left braces, e.g. \cite{MR3957824,MR3763907}. 

Clearly, the maps $\lambda_a$ and $\sigma_b$ and their interplay capture the structure of their skew brace and contain valuable information on its associated solution, as evidenced by \cite{MR4388351,castelli2023studying,colazzo2023simple}. In this context, motivated by work of  Bertram, Herzog, and Mann \cite{MR1099007}, 
we study ``common divisor'' graphs $\Lambda(A)$ and $\Theta(A)$ related to respectively
non-trivial 
$\lambda$-orbits of $A$ (see Definition~\ref{defn:graph}) and non-trivial $\theta$-orbits of $A$, which is an action from the natural semi-direct product $(A,+) \rtimes_{\lambda} (A,\circ)$ on $A$. 

In Section~\ref{comdivgraph:sec} we prove some generalities on common divisor graphs which arise from the action of a group $(G,\cdot)$ on a group $(H,\star)$, where $|G|$ and $|H|$ have the same prime divisors.  For instance, we obtain that the number of connected components is at most $2$ and that their respective diameters are at most $4$. 

As noted in Section~\ref{sec:comdivskew} and Section~\ref{sec:propsisoclinism} these general results hold in particular for $\Lambda(A)$ and $\Theta(A)$, where we focus our attention on these graphs and examine the interplay of skew brace theoretical notions on these graphs. In particular, we show that left nilpotent skew braces of nilpotent type have connected graphs of diameter at most $2$ and show that the graphs are invariant under isoclinism of skew braces, provided they are of the same size. Furthermore, we show that $\Theta(A)$ contains a homomorphic image of the common divisor graph $\Gamma(A,+)$ as introduced by Bertram, Herzog and Mann \cite{MR1099007}. Lastly, we end Section~\ref{sec:comdivskew} by noting that the $\theta$-orbits correspond to the largest homomorphic image of the solution $(A,r_A)$ that is a twist solution, i.e. $r(x,y)=(y,x)$. Moreover, it is noted that considering the $\theta$-graph of a subsolution of $(A,r_A)$ consisting of generators is an induced subgraph of $\Theta(A)$. 

In Section~\ref{examples}, we compute the graphs for all skew braces
of size $pq$ and $p^2$, where $p$ and $q$ are prime numbers. 

In the remaining sections of this article we deal with extremal cases. In Section~\ref{sec:2vertices} we examine skew braces with $\lambda$- or $\theta$-graph equal to a graph with $2$ disconnected vertices. In the former case, we show this corresponds exactly to the almost trivial skew brace on $S_3$, which is also the sole case arising in the common divisor graph of Bertram, Herzog and Mann. The skew braces with a $\theta$-graph with $2$ disconnected vertices are more numerous, but restricted to four skew braces of order $6$.

The final sections examine the skew braces with a $\lambda$- or $\theta$-graph with exactly $1$ vertex.
In Theorem~\ref{thm: onevertex:general}, we present a full classification
of finite skew braces whose $\lambda$-graph has only one vertex into $4$ infinite families. This case can not appear in the 
 conjugacy class graph of Bertram, Herzog and Mann\cite{MR1099007}, as there is no finite group with a single non-central conjugacy class. Hence, it illustrates that skew braces, even though they seem similar, can behave very differently than groups. Finally, skew braces with a one-vertex $\theta$-graph are necessarily of Abelian type. Thus Theorem~\ref{thm:onevertextheta} provides a full classification.

\section{Preliminaries}

Recall that a \emph{skew brace} is a triple $(A,+,\circ),$ where $(A,+)$ and $(A,\circ)$ are groups such that for all $a,b,c \in A$ a skew left distributivity holds $$a\circ(b+c)=a\circ b-a+a\circ c,$$ where $-a$ denotes the inverse of $a$ in $(A,+)$. Similarly, the inverse of an element $a\in A$ with respect to the operation $\circ$ is denoted by $a'.$
Following the standard terminology, if $(A, +)$ has the 
property $\chi$ (for example abelian, nilpotent,...), we say that $A$ is a skew brace of $\chi$ \emph{type}.
If a skew brace $A$ also satisfies the skew right distributivity
$$(b+c)\circ a=b\circ a-a+c\circ a,$$
$A$ is called \emph{two-sided}.
A \emph{bi-skew brace} is a skew brace $(A,+,\circ)$ 
such that also $(A,\circ,+)$ is a skew brace

Any group $G$ yields at least two skew brace structures by setting $g+h=g\circ h=gh$ or $h+g=g\circ h=gh$.
The first one is the \emph{trivial skew brace} $\Triv(G)$ on $G$. 
The second one is the \emph{almost trivial skew brace} $\op\Triv(G)$ on $G$.

Let $A$ be a skew brace. Then there is an action by automorphisms of $(A,\circ)$ on $(A,+)$ defined by
\[
\lambda_a(x)=-a+a\circ x
\]
and an action by automorphisms of $(A,+)\rtimes_{\lambda} (A,\circ)$ on $(A,+)$ defined by 
\[
\theta_{(a,b)}(c)=a+\lambda_b(c)-a.
\]
An \emph{ideal} of $A$ is a subset $I$ of $A$ such that $(I,+)$ is a normal subgroup of $(A,+),$ $(I,\circ)$ is a normal subgroup of $(A,\circ)$  and $\lambda_a(I)\subseteq I$ for all $a\in A.$ 
The \emph{socle of $A$} is the ideal $\Soc(A)=\ker\lambda\cap Z(A,+)$ and 
the \emph{annihilator of $A$} is the ideal $\Ann(A)=\Soc(A)\cap Z(A,\circ)$.

For all $a,b\in A$ we denote $a*b=\lambda_a(b)-b$.
Let $X,Y$ be subsets of $A$. Then $X*Y$ is the additive subgroup of $A$ generated by $x*y$ for all $x\in X$ and $y\in Y$.
We set also $A^{n+1}=A*A^n$ for all $n>1$ and $A^1=A$. Then $A$ is \emph{left nilpotent} if there exists a positive integer $n$ such that $A^n=\{0\}$.
The minimal positive integer $n$ such that $A^{n+1}=0$
is called the \emph{nilpotency class} of $A$.
The \emph{commutator} $A'$ of $A$ is the additive
subgroup of $A$ generated by the commutator $[A, A]_+$ of $(A,+)$ and $A^2$.

Let $A$ be a skew brace and $x\in A$. We denote the \emph{stabilizer of $x$ in $A$} by $\Stab_A(x)=\{a\in A\colon \lambda_a(x)=x\}$, or simply $\Stab(x)$ when it is clear the skew brace we are considering.
The \emph{$\lambda$-orbit of $x$ in $A$} is $\Lambda_A(x)=\{\lambda_a(x)\colon a\in A\}$, or simply $\Lambda(x)$.
Moreover, $\Fix(A)=\{a\in A\colon \lambda_x(a)=a\colon \textnormal{ for all } x\in A\}$ is the set of elements with a trivial $\lambda$-orbit.
For every $m\in\Z_{\geq 0}$ we denote by $\ell_A(m)$ (or simply by $\ell(m)$)
the number of $\lambda$-orbits of $A$ of size $m$.
By the Orbit-Stabilizer theorem, if $A$ is a finite skew brace,
then $\left|\Lambda(a)\right|=\frac{|A|}{\left|\Stab(a)\right|}$ for every $a\in A$.
Similarly, the \emph{$\theta$-stabilizer of $x$ in $A$} 
is $\Stab_\theta(x)=\{a\in A\colon \theta_a(x)=x\}$,
the \emph{$\theta$-orbit of $x$ in $A$} is $\Theta_A(x)=\{\theta_{(a,b)}(x) \colon (a,b)\in (A,+)\rtimes_{\lambda} (A,\circ)\}$, or simply $\Theta(x)$.
Moreover, the set of elements with a trivial $\theta$-orbit is
\[
\Fix_\theta(A)=\{a\in A\colon \theta_{(x,y)}(a)=a\textnormal{ for all } (x,y)\in (A,+)\rtimes_{\lambda} (A,\circ)\}
\]
and it is easy to show that $\Fix_\theta(A)=\Fix(A)\cap Z(A,+)$.
For every $m\in\Z_{\geq 0}$ we  denote by $t_A(m)$ (or simply by $t(m)$)
the number of $\theta$-orbits of $A$ of size $m$.
By the Orbit-Stabilizer theorem, for a finite skew brace $A$ and $a\in A$ $\left|\Theta(a)\right|=\frac{|A|^2}{\left|\Stab_\theta(a)\right|}$ for every $a\in A$.

We use the word \emph{graph} to denote
an undirected simple graph, i.e. 
a set of \emph{vertices} and a set of 2-subsets of vertices called \emph{edges}.
If $\{v,w\}$ is an edge we say that the vertices $v$ and $w$ are \emph{adjacent}.
Two graphs are \emph{isomorphic} if there is a bijection
between the set of vertices of two graphs such that two vertices are adjacent if and only if their images are adjacent.
A graph is called \emph{complete} if every pair of vertices are adjacent.
A \emph{path of length $r$ from $v$ to $w$} in a graph is 
a sequence of $r+1$ vertices starting with $v$ and ending with $w$ such that
consecutive vertices are adjacent.
A graph is \emph{connected} if there is a path between any two vertices.
An induced subgraph that is maximal subject to being connected is a \emph{connected component}.
If $\Gamma$ is a graph, the distance $d_{\Gamma}(v,w)$, or simply $d(v,w)$,
between two connected vertices $v$ and $w$ of $G$
is the length of a shortest path connecting them.
If $v$ and $w$ are not connected, then 
we set $d(v,w)=\infty$.
The \emph{diameter} of  a connected graph
$\Gamma$ is $d(\Gamma)=\max_{v,w}d_{\Lambda}(v,w)$.

Given a group $(G,\cdot)$ and a subset $S$
we write $\langle S\rangle_{\cdot}$
or simply $\langle S\rangle$ to indicate
the subgroup of $G$ generated by $S$.
We also write $S\leq (G,\cdot)$ (or simply $S\leq G$)
if $S$ is a subgroup of $G$. 
Moreover, given $S_1$ and $S_2$ two subsets of $G$, we denote
\[
S_1S_2=S_1\cdot S_2=\{xy\colon x\in S_1\text{ and }y\in S_2\}\quad \text{and}\quad S^{-1}_1=\{x^{-1}\colon x\in S_1\},
\]
simplifying the notation to $S_1y$ when $S_2=\{y\}$.
If we are considering an additive group $(G,+)$ we
denote $S^{-1}$ by $-S$.

\begin{rem}
\label{set ker:rem}
    Let $(G,\cdot)$ be a group and $S$ a subset of $G$. Then 
    \[
    \ker S=\{x\in G\colon Sx=S\}\leq G
    \]
    and $S\ker S =S$. So $S$ is a union of cosets of $\ker S$, thus $|\ker S|$ divides $|S|$.
\end{rem}

\section{Common divisor graph}\label{comdivgraph:sec}
In this section, we introduce a graph measuring the coprimality of non-trivial orbits of a group action.
As will be shown, these can be thought of as extensions of the graph introduced by Bertram, Herzog and Mann \cite{MR1099007}.
First, we establish relations between sizes of orbits, which play a crucial role in later results. 
Furthermore, we show that the number of connected components of these graphs is at most $2$ and we bound the maximal distance between two orbits.

\begin{convention}
    Throughout this section, we will consider two finite groups
$G$ and $H$ and an action $\alpha:G\to\Aut(H)$ by automorphisms. We will also denote by $\mathcal{O}(a)$
the orbit of $a\in H$ under the action $\alpha$.
We say that an orbit $\mathcal{O}(a)$ is non-trivial, if $|\mathcal{O}(a)|>1$.
\end{convention}

\begin{defn}
\label{defn:common divisor graph}
We define $\mathcal{C}(\alpha)$ as the graph with vertices the non-trivial orbits of the action and two different vertices $L_1,L_2$ are adjacent if $\gcd(|L_1|,|L_2|)\neq 1$.
\end{defn}

\begin{exa}
\label{exa: conjugacy class graph}
Let $G$ be a finite group and let $con:G\to\Aut(G)$ be such that $\con(g)$ is conjugation by $g$. Then $\mathcal{C}(\con)$ is the graph $\Gamma(G)$, as defined by Bertram, Herzog and Mann in~\cite{MR1099007}.
\end{exa}

Of particular importance is the case that the orders of both groups have the same prime factors. The following proposition shows that coprime orbits can be combined into a new orbit.
\begin{pro}
\label{pro: coprime orbits}
Suppose that $|G|$ and $|H|$ have the same prime factors. If $a,b\in H$ such that $|\mathcal{O}(a)|$ and $|\mathcal{O}(b)|$ are coprime, then $\Stab(a)\Stab(b)=G$ and $\mathcal{O}(a)\mathcal{O}(b)=\mathcal{O}(ab)$. 
\end{pro}

\begin{proof}
    We first prove that $\Stab(a)\Stab(b)=G$. 
    Let $p$ be a prime number 
    and $n$ be such that $p^n\mid |G|$ and $p^{n+1}\nmid |G|$.
    We claim that 
    $p^n$ divides $|\Stab(a)\Stab(b)|$. It is enough to prove that 
    $p^n$ divides either $|\Stab(a)|$ or $|\Stab(b)|$. If $p\nmid|\mathcal{O}(a)|$, 
    then, since $p^n$ divides $|G|=|\Stab(a)||\mathcal{O}(a)|$, it follows that $p^n\mid |\Stab(a)|$.
    If $p\mid|\mathcal{O}(a)|$, then, since $|\mathcal{O}(a)|$ and $|\mathcal{O}(b)|$ are coprime, $p\nmid|\mathcal{O}(b)|$ and the claim follows. 
    
    Now we prove that $\mathcal{O}(a)\mathcal{O}(b)=\mathcal{O}(ab)$. 
    Since $G$ acts by automorphisms, 
    $\mathcal{O}(a)\mathcal{O}(b)\supseteq \mathcal{O}(ab)$.
    For the reverse inclusion, let $(g\cdot a)(h\cdot b)\in \mathcal{O}(a)\mathcal{O}(b)$.
    Since $h^{-1}g\in G=\Stab(a)\Stab(b)$, there exist $x\in\Stab(a)$ and $y\in\Stab(b)$ 
    such that $gx=hy$. Thus 
    \[
    (g\cdot a)(h\cdot b)=(gx\cdot a)(hy\cdot b)=(gx\cdot a)(gx\cdot b)=gx\cdot(ab),
    \]
    which shows the result.
\end{proof}
Proposition~\ref{pro: coprime orbits} allows us to draw a few intermediate results
on the products of orbits that lie sufficiently far apart in the graph.
\begin{lem}
\label{lem: vertices at distance >2}
Suppose that $|G|$ and $|H|$ have the same prime factors.
Let $L_1,L_2$ be vertices of $\mathcal{C}(\alpha)$ 
such that $d(L_1,L_2)\geq 3$ and $|L_2|>|L_1|$. Then 
    \begin{enumerate}
        \item $L_1L_2,L_2L_1,L_1L_2^{-1}$ and $L_2L_1^{-1}$ are orbits;
        \item $|L_2L_1^{-1}|=|L_2|$;
        \item $L_2L_1^{-1}L_1=L_2$;
        \item $1<|\langle L_1^{-1}L_1\rangle|$ divides $|L_2|$.
    \end{enumerate}
\end{lem}

\begin{proof}
    The first claim follows from Proposition~\ref{pro: coprime orbits}. 

    We now prove (2). Clearly $|L_2L_1^{-1}|\geq |L_2|$.  
    By using Proposition~\ref{pro: coprime orbits} and 
    that $|\Stab(a)\cap\Stab(b)|$ divides $|\Stab(ab^{-1})|$ for any $a\in L_1$ and $b\in L_2$, 
    it follows that $|L_2L_1^{-1}|$ divides $|L_1||L_2|$. If $|L_2L_1^{-1}|>|L_2|$, since $|L_1|$ and $|L_2|$ are coprime, $\gcd(|L_2L_1^{-1}|,|L_2|)>1$ and $\gcd(|L_2L_1^{-1}|,|L_1|)>1$. Thus $d(L_1,L_2)\leq 2$, a contradiction. 
    
    We now prove (3). By (1) and (2), $|L_1|$ and $|L_2|=|L_2L_1^{-1}|$ are coprime. 
    By Proposition~\ref{pro: coprime orbits}, $L_2L_1^{-1}L_1$ is an orbit. Now the claim follows from $L_2\subseteq L_2L_1^{-1}L_1$.
        
    Finally, we prove (4). Let $K=\langle L_1^{-1}L_1\rangle$.
    Since $|L_1|>1$, then $|K|>1$. 
    Since $L_1^{-1}L_1$ is closed under inverses in $H$ and $L_2L_1^{-1}L_1=L_2$, 
    \[
    L_2L_1^{-1}L_1L_1^{-1}L_1=L_2L_1^{-1}L_1=L_2
    \]
    and hence $L_2k=L_2$ for all $k\in K$. 
    Thus $K\leq \ker L_2$ and hence $|K|$ divides~$|L_2|$.
\end{proof}
The following proposition shows that if there exist three non-trivial orbits of different sizes,
then the orbit of intermediate size should lie close to one of either other orbits.
\begin{pro}
\label{pro: 3 vertices distance}
    Suppose that $|G|$ and $|H|$ have the same prime factors.
    Then it is not possible to find three vertices $L_1,L_2,L_3$ of $\mathcal{C}(\alpha)$ 
    such that $|L_3|>|L_2|>|L_1|$ and $d(L_1,L_2)\geq 3$ and $d(L_2,L_3)\geq 3$.
\end{pro}

\begin{proof}
    Suppose on the contrary that $L_1, L_2$ and $L_3$ are vertices of $\mathcal{C}(\alpha)$ such that $|L_3|>|L_2|>|L_1|$ and $d(L_1,L_2)\geq 3$ and $d(L_2,L_3)\geq 3$. 
    
    By Lemma~\ref{lem: vertices at distance >2},
    $|\langle L_1^{-1}L_1\rangle|>1$ and it divides $|L_2|$. Moreover, $$L_2L_1^{-1}L_1=L_2 \textnormal{ and } L_3L_2^{-1}L_2=L_3.$$
    Therefore
    \[
     L_3L_1^{-1}L_1=(L_3L_2^{-1}L_2)L_1^{-1}L_1=L_3L_2^{-1}(L_2L_1^{-1}L_1)=L_3L_2^{-1}L_2=L_3.    
    \]
    
    Thus $\langle L_1^{-1}L_1\rangle$ is a subgroup of $\ker L_3$ and $|\langle L_1^{-1}L_1\rangle|$ divides $|L_3|$. 
    Hence, $|\langle L_1^{-1}L_1\rangle|$ is a common divisor of $|L_3|$ and $|L_2|$. Thus, $d(L_2,L_3)=1$, which is a contradiction. 
\end{proof}
Finally, the following theorems combine the previous results to limit the number of connected components to two and the diameter of each component to four. 
\begin{thm}
\label{thm:connected_components}
Suppose that $|G|$ and $|H|$ have the same prime factors.
Then the number of connected components of
$\mathcal{C}(\alpha)$ is at most $2$.
\end{thm}
\begin{proof}
Suppose that there exist three connected components. Let $L_1,L_2,L_3$ be three vertices belonging to different connected components with, without loss of generality, $|L_3|>|L_2|>|L_1|>1$. 
Then 
\[
\gcd(|L_1|,|L_2|)=\gcd(|L_1|,|L_3|)=\gcd(|L_2|,|L_3|)=1.
\]
By Lemma~\ref{lem: vertices at distance >2}, $|\langle L_1^{-1}L_1\rangle_+|>1$ and
divides 
$|L_3|$ and $|L_2|$, a contradiction. 
\end{proof}

\begin{thm}
\label{thm:diameter}
Suppose that $|G|$ and $|H|$ have the same prime factors.
Then the diameter of every connected component of $\mathcal{C}(\alpha)$ is at most four.
\end{thm}

\begin{proof}
    Let $L_1$ and $L_2$ be in the same connected component and such that $d(L_1,L_2)\geq 5$.
    Then there exists a vertex $L_3$ such that $d(L_3,L_1)=3$ and $d(L_3,L_2)=2$. 
    Without loss of generality, we may assume that $|L_1|<|L_2|$. 
    If $|L_1|<|L_3|$, then by Lemma~\ref{lem: vertices at distance >2}, $|\langle L_1^{-1}L_1\rangle|>1$ and it 
    divides both $|L_2|$ and $|L_3|$. 
    Thus $\gcd(|L_2|,|L_3|)>1$, a contradiction
    to $d(L_3,L_2)=2$. This implies that 
    $|L_3|<|L_1|<|L_2|$, a contradiction to Proposition~\ref{pro: 3 vertices distance}.
\end{proof}

\section{Common divisor graphs for skew braces and their subsolutions}\label{sec:comdivskew}
In this section, we define two particular instances of common divisor graphs on skew braces and handle some examples. Secondly, we note that this allows to define a common divisor graph of an injective solution of the Yang-Baxter equation. In particular, the vertices of this graph, joint with the trivial orbits, can be characterized with a property of the set-theoretic solution.

\begin{defn}
\label{defn:graph}
    Let $A$ be a finite skew brace. 
    We define the $\lambda$-graph of $A$ as the graph $\Lambda(A)=\mathcal{C}(\lambda)$
    and the $\theta$-graph of $A$
    as the graph $\Theta(A)=\mathcal{C}(\theta)$.
\end{defn}

Moreover, the $\lambda$-graph and the $\theta$-graph of a skew brace generalize the graph introduced by Bertram, Herzog and Mann \cite{MR1099007}.

Observe that $\lambda$ is an action of $G=(A,\circ)$ on $H=(A,+)$ by automorphisms
and $\theta$ is an action of
$G=(A,+)\rtimes_\lambda (A,\circ)$ on $H=(A,+)$ by automorphisms.
So in both cases $|G|$ and $|H|$ have the same prime
factors and we can apply all the results of the 
previous section. In particular, the following 
proposition follows from Theorems~\ref{thm:connected_components} and~\ref{thm:diameter}
and is mentioned to emphasize the results.
    
\begin{pro}
    Let $A$ be a finite skew brace. Then, the $\lambda$- and $\theta$-graphs of $A$ have at most two connected components, each of diameter at most four.
\end{pro}

If $A$ is a left nilpotent skew brace of nilpotent type, the previous result can be strengthened.

\begin{thm}
    Let $A$ be a finite skew brace.
    If $A$ is left nilpotent of nilpotent type, then 
    $\Theta(A)$ and $\Lambda(A)$ are connected of diameter at most two.
\end{thm}
\begin{proof}
    Since $A$ is left nilpotent of nilpotent type,
    by~\cite[Theorem 4.8, Corollary 4.3]{MR3957824}
    $A$ is the direct product
    $A=A_1\times\dots\times A_n$, 
    where $A_i$ is a $p_i$-brace with additive group 
    the Sylow $p_i$ subgroup of $(A,+)$.

    Let $a,b$ elements of $A\setminus \Fix(A)$ such that $\Lambda(a)$ and $\Lambda(b)$ are not connected in $\Lambda(A)$. 
    We will prove the result by showing that there exists
    $c\in A$ such that $\Lambda(c)$
    is adjacent to $\Lambda(a)$ and to $\Lambda(b)$
    in $\Lambda(A)$.
    First of all we can write $a=(a_1,\dots, a_n)$
    and $b=(b_1,\dots, b_n)$ for some
    $a_i,b_i\in A_i$ for $i\in\{1,\dots,n\}$.
    Then the $\lambda$-orbits are $\Lambda(a)=\Lambda_{A_1}(a_1)\times\dots\times\Lambda_{A_n}(a_n)$ 
    and $\Lambda(b)=\Lambda_{A_1}(b_1)\times\dots\times\Lambda_{A_n}(b_n)$.
    Since $a,b\not\in \Fix(A)$ and $\Lambda(a)$ and
    $\Lambda(b)$ are not connected in $\Lambda(A)$,
    there exist $i,j\in\{1,\dots,n\}$ such that $i\neq j$ and
    $\Lambda_{A_i}(a_i)$ and $\Lambda_{A_j}(b_j)$ are non trivial, i.e. $p_i$ divides $|\Lambda_{A_i}(a_i)|$
    and $p_j$ divides $|\Lambda_{A_j}(b_j)|$.
    Now we can consider the element $c=(c_1,\dots,c_n)\in A_1\times\dots\times A_n$,
    with $c_i=a_i$, $c_j=b_j$ and $c_k=0$ for all $k\not\in\{i,j\}$.
    It is clear that $|\Lambda(c)|=|\Lambda_{A_i}(a_i)||\Lambda_{A_j}(b_j)|$ and hence that $\Lambda(c)$
    is adjacent to $\Lambda(a)$ and $\Lambda(b)$
    in $\Lambda(A)$.

    In the same way, one can show the result for $\Theta(A)$.
\end{proof}

The converse of the previous theorem 
is not true in general, as we can find finite non left nilpotent skew braces
with connected $\lambda$- and $\theta$-
 graphs of diameter less than or equal to two.

For every $n\in \Z_{\geq 0}$, we denote
by $\Z/n\Z$ the ring of integers modulo $n$.
\begin{exa}
    Let $A$ be the skew brace 
    with additive group $\Z/3\Z\times\Z/2\Z$ and
    multiplication $(n,m)\circ(s,t)=(n+(-1)^ms,m+t)$.
    Then $A$ is not left nilpotent but 
    $\Theta(A)=\Lambda(A)$ is the complete graph on two vertices.
\end{exa}
    
It is also important to note that
if we drop the hypothesis of $A$ being of nilpotent type, the result of the previous theorem doesn't hold.
\begin{exa}
    The trivial skew brace $\Triv(S_3)$ is clearly left nilpotent but has a disconnected $\theta$-graph.
\end{exa}

\begin{pro} \label{pro: subgraph}
    Let $A$ be a skew brace and $a\in A$. 
    Denote $\Con(a)$ the conjugacy class of $a$ in $(A,+)$.
    Then, $\Con(a),\Lambda(a)\subseteq\Theta(a)$.
    
    Moreover, $\Theta(a)=\bigcup_{x\in A}(\Con(\lambda_x(a))$ 
    is a union of conjugacy classes of $(A,+)$ of the same size.
    In particular, the graph $\Theta(A)$ contains a homomorphic image of $\Gamma(A,+)$.
\end{pro}
\begin{proof}
    For all $x\in A$, we have that 
    $\lambda_x(a)=\theta_{(0,x)}(a)$ and $x+a-x=\theta_{(x,0)}(a)$. In particular, $\Con(a)$ and $\Lambda(a)$ are contained in $\Theta(a)$. As, for any $x,y,a \in A$,  $$\lambda_x(y+a-z)= \lambda_x(y)+\lambda_x(a)-\lambda_x(y),$$ it follows that $\lambda_x(\Con(a)) \subseteq \Con(\lambda_x(a))$. Note that the reverse inclusion follows from applying $\lambda_{x'}$, where $x'$ denotes the inverse of $x$ in $(A,\circ)$. Thus, the conjugacy classes of $a$ and $\lambda_x(a)$ have the same size. 
    Moreover, it holds that the union $ \bigcup_{x \in A} \Con(\lambda_x(a))$ is closed under the $\theta$-action. As it is clearly contained in $\Theta(a)$, equality must ensue. Finally, the map $\varphi: \Gamma(A,+) \rightarrow \Theta(A)$ given by $\varphi(\Con(a)) = \Theta(a)$ is clearly well-defined. Since $\Theta(a)$ is the disjoint union of conjugacy classes of the same size, it follows that $|\Theta(a)|$ is a multiple of $|\Con(a)|$. In particular, if two vertices in $\Gamma(A,+)$ are connected and not identified by $\varphi$, then their images under $\varphi$ are connected.  
\end{proof}

Note that the subgraph obtained in Proposition~\ref{pro: subgraph} is not necessarily an induced subgraph. 
\begin{exa}
  Let $(A,\circ)=\langle s\colon s^{12}=0\rangle$ be the cyclic group of order 12
  and $(A,+,\circ)$ the skew brace with addition $s^m+s^n=s^{m+(-1)^mn}$ for all $m,n\in\Z$.
  The $\lambda$-action is given by $\lambda_{s^m}(s^n)=s^{n+(1+(-1)^{n+1}m)}$,
  thus 
  \[
  Z(A,+)=\{0,s^6\}\subset \Fix(A)=\{s^{2i}\colon i\in\Z\}\text{ and }\Fix_{\theta}(A)=Z(A,+)=\{0,s^6\}.
  \]
  $A$ as only one non-trivial $\lambda$-orbit of size six: $\Lambda(s)=\{s^{2i+1}\colon i\in\Z\}$.
  As for the conjugation action, $\con_{s^n}(s^m)=s^{(-1)^n(-n+m+(-1)^mn)}$, so
  \[
  \Con(s)=\{s,s^5,s^9\},\quad \Con(s^3)=\{s^3,s^7,s^{11}\}
  \]
  \[
  \Con(s^2)=\{s^2,s^{10}\},\quad \Con(s^4)=\{s^4,s^8\}
  \]
  By Proposition~\ref{pro: subgraph} we have that 
  $\Theta(s^n)=\bigcup_{m\in \Z}(\Con(\lambda_{s^m}(s^n))$ for all $n\in\Z$, hence
  $\Theta(s^{2i})=\Con(s^2) $ and $\Theta(s^{2i+1})=\Con(s)\cup\Con(s^3)$
  for all $i\in \Z$.
  Therefore the graphs related to $A$ are the following:
  
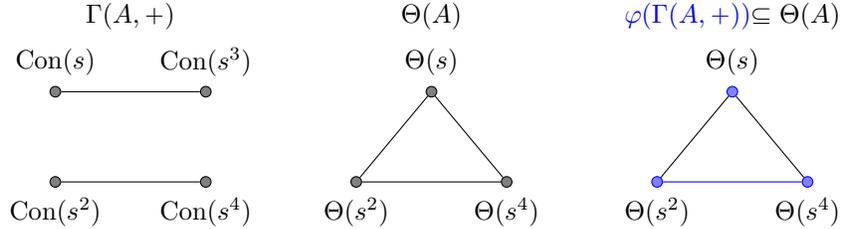
\begin{center}
  \begin{tikzpicture}
  \node (gamma) at (1,5)  {$\Gamma(A,+)$};
  \node[circle,draw, fill=black!50,inner sep=0pt, minimum width=4pt] (s) at (0,4) {};
  \node[circle,draw, fill=black!50,inner sep=0pt, minimum width=4pt] (s3) at (2,4) {};
  \node[circle,draw, fill=black!50,inner sep=0pt, minimum width=4pt] (s2) at (0,2.8)  {};
  \node[circle,draw, fill=black!50,inner sep=0pt, minimum width=4pt] (s4) at (2,2.8) {};
  \node (slab) at (0,4.4) {$\Con(s)$};
  \node (s3lab) at (2,4.4)  {$\Con(s^3)$};
  \node (s2lab) at (0,2.4)  {$\Con(s^2)$};
  \node (s4lab) at (2,2.4) {$\Con(s^4)$};
  \node (theta) at (5,5)  {$\Theta(A)$};
  \node[circle,draw, fill=black!50,inner sep=0pt, minimum width=4pt] (ts) at (5,4) {};
  \node[circle,draw, fill=black!50,inner sep=0pt, minimum width=4pt] (ts2) at (4,2.8)  {};
  \node[circle,draw, fill=black!50,inner sep=0pt, minimum width=4pt] (ts4) at (6,2.8) {};
  \node (tslab) at (5,4.4) {$\Theta(s)$};
  \node (ts2lab) at (4,2.4)  {$\Theta(s^2)$};
  \node (ts4lab) at (6,2.4) {$\Theta(s^4)$};
  \node (theta) at (9,5)  {\textcolor{blue}{$\varphi(\Gamma(A,+))$}$\subseteq\Theta(A)$};
  \node[circle,draw=blue, fill=blue!50,inner sep=0pt, minimum width=4pt] (t's) at (9,4) {};
  \node[circle,draw=blue, fill=blue!50,inner sep=0pt, minimum width=4pt] (t's2) at (8,2.8)  {};
  \node[circle,draw=blue, fill=blue!50,inner sep=0pt, minimum width=4pt] (t's4) at (10,2.8) {};
  \node (t'slab) at (9,4.4) {$\Theta(s)$};
  \node (t's2lab) at (8,2.4)  {$\Theta(s^2)$};
  \node (t's4lab) at (10,2.4) {$\Theta(s^4)$};
  \foreach \from/\to in {s/s3,s2/s4,ts/ts2,ts/ts4,ts2/ts4,t's/t's2,t's/t's4}
    \draw (\from) -- (\to);
  \draw[blue] (t's2) -- (t's4);
\end{tikzpicture}  
\end{center}
Therefore $\Theta(A)$ contains a homomorphic image of $\Gamma(A,+)$
which is not an induced subgraph of $\Theta(A)$.
\end{exa}

The class of skew braces with empty graphs can be easily characterized.
\begin{pro}
    Let $A$ be a finite skew brace. Then $\Lambda(A)$
    has no vertices if and only if 
    $A$ is a trivial skew brace. 
    Moreover $\Theta(A)$ has no vertices if and only if 
    $A$ is a trivial skew brace of abelian type.
\end{pro}

For a finite skew brace of abelian type, its $\lambda$-graph and its $\theta$-graph coincide.

\begin{rem}
    Let $G$ be a finite group.
    Then the trivial skew brace on $G$ has an empty $\lambda$-graph and the $\theta$-graph is equal to the graph $\Gamma(G)$.
    Moreover
    \begin{align*}
        \Gamma(G)&=\Lambda(\op\Triv(G))=\Theta(\op\Triv(G))\\
        &=\Gamma(G^{\op})=\Lambda(\op\Triv(G^{\op}))=\Theta(\op\Triv(G^{\op})).
    \end{align*}
\end{rem}

However, there exist graphs that can not appear as the conjugacy class graph of a group but do appear as the $\theta$- or $\lambda$-graph of a skew brace. A particular example that will be completely characterized in Section~\ref{sec:onevertex} is the graph with a single vertex.

The $\theta$-graph carries some information about the associated set-theoretic solution to a skew brace. In particular, it governs the largest homomorphic image which is a twist solution, i.e. a solution $(X,r)$ with  $r(x,y)=(y,x)$ for all $x,y \in X$. Recall that the associated solution of a skew brace $A$ is given by $r_A(a,b) = (\lambda_a(b),\rho_b(a))$, where $a,b \in A$ and $$ \rho_b(a) = \lambda_a(b)'\circ a \circ b$$ with $c'$ denoting the inverse of $c \in (A,\circ)$. 

\begin{pro}
Let $(A,+,\circ)$ be a skew brace. Denote $A/\theta$ the set of orbits under the action of $\theta$. Denote $(A,r_A)$ the associated solution of $(A,+,\circ)$. Then there exists a homomorphism of solutions $\varphi: (A,r_A) \rightarrow (A/\theta,\tau)$, where $\tau(x,y)=(y,x)$. Moreover, every homomorphism of solutions $(A,r_A)\rightarrow (Y,s)$, where $s$ is the twist on $Y$, can be factored through $\varphi$.
\end{pro}
\begin{proof}
Define $\varphi:A\to A/\theta$ as the map $\varphi(a)=\Theta(a)$
for all $a\in A$.
Recall that $r_A(a,b)=(\lambda_a(b),\rho_b(a))$, for all $a\in A$.
As $a\circ b=a+\lambda_a(b)$ and $\lambda_{a'}(a)=-a'$ for every $a,b\in A$,
\begin{align*}
    \rho_b(a)&=\lambda_a(b)'\circ a \circ b
    =\lambda_a(b)'+\lambda_{\lambda_a(b)'}(a)+\lambda_{\lambda_a(b)'}\lambda_a(b)\\
    &=\lambda_a(b)'+\lambda_{\lambda_a(b)'}(a)-\lambda_a(b)'
    =\Theta_{(\lambda_a(b)',\lambda_a(b)')}(a).
\end{align*}
for every $a,b\in A$.
Since $\lambda_a(b)\in \Theta(b)$ and $\rho_b(a)\in \Theta(a)$, we have that $\varphi$ is a morphism of solution: for all $a,b\in A$
\[
(\varphi\times\varphi)(r_A(a,b))=(\Theta(\lambda_a(b)),\Theta(\rho_b(a))) =
(\Theta(b),\Theta(a)) =
\tau(\varphi(a),\varphi(b)).
\]
Let now $(Y,s)$ be a twist solution and $\psi:(A,r_A)\to (Y,s)$ be a morphism of solutions,
i.e. $s\circ(\psi\times\psi)=(\psi\times\psi)\circ r_A$.
This means that 
\begin{equation}
\label{eq: psi morphism}
    \psi(\lambda_a(b))=\psi(b)\text{ and }
\psi(\rho_b(a))=\psi(a) \text{ for all }a,b\in A.
\end{equation}
Therefore if $c\in \Theta(a)$, then $c=\Theta_{(x,y)}(a)$ for some
$x,y\in A$. 
A direct verification shows that
\begin{equation}\label{eq:thetasomething}\Theta_{(x,y)}(a)=\lambda_{y\circ (-\lambda_{y'}(x))}(\rho_{-\lambda_{a'\circ y'}(x)}(a)), \end{equation}
and so, using Equation ~\eqref{eq: psi morphism}, we get that
\[
\psi(c)=\psi(\lambda_{y\circ (-\lambda_{y'}(x))}(\rho_{-\lambda_{a'\circ y'}(x)}(a)))=\psi(\rho_{-\lambda_{a'\circ y'}(x)}(a))=\psi(a).
\]
Thus we can define a map $\Bar{\psi}:A/\theta\to Y$ as 
$\Bar{\psi}(\Theta(a))=\psi(a)$ for all $\Theta(a)\in A/\theta$.
Moreover, it is easy to see that $\Bar{\psi}$ is a morphism of solutions and, by definition, $\Bar{\psi}\circ \varphi=\psi$.
\end{proof}

Furthermore, a set of generators $X$ of a finite skew brace $A$ that is a subsolution gives rise to an induced subgraph. In particular, this can be applied to the image of a bijective non-degenerate solution $(X,r)$ considered as the generators of $\mathcal{G}(X,r) = \left< (\lambda_x,\rho_x^{-1})\mid x \in X\right>$,
as reformulated as $\mathcal{G}_{\lambda,\rho}$ in \cite{MR4457900} and first defined by Bachiller in \cite{bachiller2018solutions}.

\begin{pro}
    Let $(A,+,\circ)$ be a finite skew brace with associated solution $(A,r_A)$. Let $(X,r)$ be a subsolution of $(A,r_A)$ such that $X$ generates $A$. Then, the induced subgraph obtained by considering the non-trivial $\theta$-orbits that intersect $X$ coincides with the common divisor graph with vertices the non-trivial orbits of $X$ under $\theta$-action of $A$ where edges are drawn if the greatest common divisor is at least two. 
\end{pro}
\begin{proof}
    Let $x \in X$ be contained in a non-trivial $\theta$-orbit $\mathcal{O}$ of $A$. Then, as $X$ is a set of generators and a subsolution, it follows from Equation \eqref{eq:thetasomething} that $X$ is closed under $\theta$. Thus, $\mathcal{O} \subseteq X$. This shows that every non-trivial orbit that contains an element of $X$ is fully within $X$. Hence, the non-trivial $\theta$-orbits of $X$ coincide with the non-trivial $\theta$-orbits containing an element of $X$, which shows the result.
\end{proof}

\section{Basic properties and isoclinism of skew braces}\label{sec:propsisoclinism}
In this section, we characterize the ideal $A^2$ of a skew brace in terms of non-trivial $\lambda$-orbits. Furthermore, we establish several counting arguments, which are crucial for the remainder of the paper. We finish the section by investigating the relation between isoclinism of skew braces and their graphs. 
\begin{pro}
\label{sizeA2:pro}
Let $A$ be a finite skew brace and $\Lambda(x_1),\dots,\Lambda(x_k)$ be the non-trivial $\lambda$-orbits. Then
\[
A^2=\Big\langle \bigcup_{i=1}^k\Lambda(x_i)-\Lambda(x_i)\Big\rangle_+.
\]
In particular, $|A^2|\geq |\Lambda(x_i)|$ for all $i\in\{1,\dots,k\}$.
\end{pro}

\begin{proof}
If $b\in\Fix(A)$, then $\lambda_a(b)-b=0$ for all $a\in B$. Thus 
\[
A^2=\langle \lambda_a(b)-b\colon a\in A, b\in A\setminus\Fix(A)\rangle_+.
\]
For every $b\in A\setminus\Fix(A)$ there exist $i\in\{1,\dots, k\}$ and $c\in A$ such that $b=\lambda_c(x_i)$. Thus every generator of $A^2$ is of the form
\[
\lambda_a(b)-b=\lambda_{a\circ c}(x_i)-\lambda_c(x_i)\in \Lambda(x_i)-\Lambda(x_i).
\]
The converse is easily seen, as $$ \lambda_c(x_i)-\lambda_a(x_i) = \lambda_{c\circ a'}(\lambda_a(x_i))-\lambda_a(x_i) \in A^2,$$ which shows the result.
\end{proof}

Similarly one characterizes $A'$ using non-trivial $\theta$-orbits.

\begin{pro}
\label{sizeA':pro}
Let $A$ be a skew brace and $\Theta(y_1),\dots,\Theta(y_h)$ be the non-trivial $\theta$-orbits. Then
\[
A'=\Big\langle \bigcup_{j=1}^h\Theta(y_j)-\Theta(y_j)\Big\rangle_+.
\]
In particular, $|A'|\geq |\Theta(y_j)|$ for all $j\in\{1,\dots,h\}$.
\end{pro}

\begin{proof}
Let $\Lambda(x_1),\dots,\Lambda(x_k)$ be the non-trivial $\lambda$-orbits of $A$.
Since $\Fix_\theta(A)$ is contained in $\Fix(A)$,
$\Theta(x_i)$ is non-trivial  for every $i\in\{1,\dots,k\}$, hence there exist $j\in\{1,...,h\}$ such that
$\Lambda(x_i)\subseteq \Theta(x_i)=\Theta(y_j)$.
Therefore, using Proposition~\ref{sizeA2:pro}, we obtain that
\[
A^2=\Big\langle \bigcup_{i=1}^k\Lambda(x_i)-\Lambda(x_i)\Big\rangle_+\subseteq \Big\langle \bigcup_{j=1}^h\Theta(y_j)-\Theta(y_j)\Big\rangle_+.
\]
Moreover $[a,b]_+=\theta_{(a,0)}(b)-b\in \Theta(b)-\Theta(b)$
for every $a,b\in A$.
Therefore, if $b\in \Fix_\theta(A)$ then $[a,b]_+=0$ and 
if $b\in A\setminus\Fix_\theta$ there exist $j\in\{1,\dots, h\}$
such that $\Theta(b)=\Theta(y_j)$. 
Therefore  $[A,A]_+\subseteq \Big\langle \bigcup_{j=1}^h\Theta(y_j)-\Theta(y_j)\Big\rangle_+$ and so \[
A'=\langle [A,A]_+,A^2\rangle\subseteq \Big\langle \bigcup_{j=1}^h\Theta(y_j)-\Theta(y_j)\Big\rangle_+.
\]
Vice versa 
\[
\theta_{(a,b)}(y_j)-\theta_{(c,d)}(y_j)=[a,\lambda_b(y_j)]_++b*y_j-d*y_j+[c,\lambda_d(y_j)]_+\in A'.
\]
Hence $A'\supseteq \Big\langle \bigcup_{j=1}^h\Theta(y_j)-\Theta(y_j)\Big\rangle_+$.
\end{proof}

In the following propositions, we set up general numerical results
on the number of orbits of a particular size of an action of a group 
on another group by automorphisms.

\begin{pro}
\label{fix:pro}
Let $G$ and $H$ be finite groups such that $G$ acts on $H$ by automorphisms and for all $m\in\Z_{\geq 0}$ let $\mathbf{c}(m)$ be the number
of orbits of size $m$.
Then $\left|H^G\right|$ divides $m\mathbf{c}(m)$ for all $m\in\Z_{\geq 0}$.
\end{pro}

\begin{proof}
    Let $m\geq1$. 
    Since, for all $h\in H$ and $x\in H^G$,
    $$\Stab(hx)=\Stab(h) 
    \textnormal{ and } |\mathcal{O}(hx)|=|\mathcal{O}(h)|,$$
    it follows that 
    \[
    \bigsqcup\limits_{|\mathcal{O}(a)|=m}\left(a H^G\right)
    =\bigsqcup\limits_{|\mathcal{O}(a)|=m}\mathcal{O}(a).
    \]
    Hence, the union of orbits of size $m$ is 
    a union of some cosets of $ H^G$. 
    Thus $k\left| H^G\right|=m\mathbf{c}(m)$ for some $k\in\Z_{\geq 0}$. 
\end{proof}

\begin{pro}
\label{pro:p-groups orbits}
  Let $p$ be a prime number and 
  $G$ and $H$ be a finite $p$-groups.
  Let $G$ be a group acting on $H$ by automorphisms.
  Then $p-1$ divides $\mathbf{c}(p^m)$ for all $m\in\Z_{\geq 0}$. 
\end{pro}
\begin{proof}
Let $m\in\Z_{\geq0}$. We prove that $p-1$ divides $\mathbf{c}(p^m)$. $(\Z/p\Z)^{\times}$ acts
on the set of orbits of size $p^m$ by $f_j(\mathcal{O}(a))=\mathcal{O}(a^j)$
where by $a^j$ we mean $a\cdot \ldots \cdot a,$ $j$ times.
It is well-defined since $|\mathcal{O}(a^j)|=|\mathcal{O}(a)|$
and $f_j$ is a bijection for all $j\in(\Z/p\Z)^{\times}$.
In fact $f_j(\mathcal{O}(a))=f_j(\mathcal{O}(b))$ if and only if
there exists $x\in G$ such that $b^j=x\cdot a^j$.
Taking $s\in\Z_{\geq 0}$ such that $sj\equiv 1\pmod{p^d}$,
where $p^d$ is the order of both $a$ and $b$, we get
\[
b=b^{sj}=(x\cdot(a^j)^s=x\cdot (a^{sj})=x\cdot a.
\]
Finally, if $f_j(\mathcal{O}(a))=\mathcal{O}(a)$, then $j=1$.
Indeed if $f_j(\mathcal{O}(a))=\mathcal{O}(a)$,
then there exists $x\in G$ 
of multiplicative order $p^\beta$ such that $x\cdot a=a^j$.
Thus 
 \[
a=x^{p^\beta}\cdot a=a^{j^{p^\beta}}
\]
and $j^{p^\beta}\equiv 1\pmod{p^\alpha}$,
where $p^\alpha$ is the order of $a$.
Hence $j^{p^\beta}\equiv 1\pmod{p}$ and
the order of $j$ in $\Z/p\Z^\times$ divides $p^\beta$, so $j=1$.
Thus $p-1$ divides the size of every orbit of this new action.
Therefore, being a union of orbits, the set of all orbits of size $p^m$ 
has cardinality divisible by $p-1$.
\end{proof}

\begin{cor}
     Let $p$ be a prime number and 
$G$ and $H$ be finite $p$-groups such that
$G$ acts on $H$ by automorphisms.
Denote $p^s$ the minimal size of a non-trivial orbits. Then, one of the following holds \begin{enumerate}
         \item $\left| H^G\right|=p^s$ and $\mathbf{c}(p^s)\equiv p-1 \pmod{p(p-1)},$
         \item $\left| H^G\right|>p^s$ and $\mathbf{c}(p^s)\equiv 0 \pmod{p(p-1)}$.
     \end{enumerate}      
\end{cor}
\begin{proof}
Let $\left|H^G\right|=p^z$ and $|H|=p^n$.
By the class equation
$$p^n=p^z+\sum_{m\geq s}\mathbf{c}(p^m)p^m,$$
we get $s\leq z$ and $p^{n-s}=p^{z-s}+\mathbf{c}(p^s)+\sum_{m> s}\mathbf{c}(p^m)p^{m-s}$.
If $s=z$, then, as $p-1$ divides $\mathbf{c}(p^s)$ by Proposition~\ref{pro:p-groups orbits}, it holds that 
\[
\mathbf{c}(p^s)\equiv p-1\pmod{p(p-1)}.
\]
Otherwise, $s<z$ so $p$ divides $\mathbf{c}(p^s).$ Thus, by Proposition~\ref{pro:p-groups orbits}, one obtains
\[\mathbf{c}(p^s)\equiv 0\pmod{p(p-1)}. \qedhere
\]
\end{proof}

We end this section by examining the link between isoclinism and the $\theta$- and $\lambda$-graphs.
In particular, we show that both graphs are invariant under isoclinism if the skew braces are of the same size.

\begin{defn}[\cite{LeVe}]
\label{defn:isoclinism}   
 Two skew braces $A$ and $B$ are \emph{isoclinic} if there are two isomorphisms $\xi\colon A/\Ann(A)\to B/\Ann(B)$ and $\delta\colon A'\to B'$ such that the following diagram commutes
\begin{center}
\begin{tikzcd}
A' \arrow[r, "\phi_+^A",leftarrow] \arrow[d, "\delta"]
& A/\Ann(A)\times A/\Ann(A) \arrow[r,"\phi_*^A"]\arrow[d, "\xi\times\xi" ] & A'\arrow[d, "\delta"]\\
B'\arrow[r, "\phi_+^B",leftarrow] & B/\Ann(B)\times B/\Ann(B) \arrow[r,"\phi_*^B"]& B'
\end{tikzcd}
\end{center}
where $\phi_+^A,\phi_*^A\colon A/\Ann(A)\times A/\Ann(A)\to A'$ are defined by $\phi_+^A(\overline{a},\overline{b})=[a,b]_+$ and $\phi_*^A(\overline{a},\overline{b})=a*b$ for all $a,b\in A.$
\end{defn}

\begin{pro}(\cite[Proposition 3.10]{LeVe})
    Let $A$ and $B$ be two isoclinic skew braces. 
Then $(A,+)$ and $(B,+)$ are isoclinic groups, 
$(A,\circ)$ and $(B,\circ)$ are isoclinic groups.
\end{pro}

Let $A$ be a skew brace. We can define 
an action 
\[
\overline{\rho}:(A,\circ)/\Ann(A)\to \Aut\left((A,+)/\Ann(A)\right),
\text{ by }\overline{\rho}_{\bar{a}}(\bar{b})=a\circ b-a.
\]
Moreover, we also have an action
\[
\alpha_A:(A/\Ann(A),+)\rtimes_{\overline{\rho}}(A/\Ann(A),\circ)\to \Aut(A,+)
\]
given by $\alpha_{(\bar{a},\bar{b})}(c)=a+\overline{\rho}_{\bar{b}}(\bar{c})-a$.

If  $H$ is a subgroup of $(A/\Ann(A),+)\rtimes_{\overline{\rho}}(A/\Ann(A),\circ)$,
we call \emph{$H$-orbit} the orbit of an element of $A$
under the action $\alpha$ restricted to the subgroup $H$.

\begin{thm}(\cite[Theorem 3.19]{LeVe})
\label{thm:isoclinc orbits}
Let $A$ and $B$ be finite isoclinic skew braces. 
Let $H$ be a subgroup of $(A/\Ann(A),+)\rtimes_{\overline{\rho}}(A/\Ann(A),\circ)$
and $K$ the corresponding (via isoclinism) subgroup of $(B/\Ann(B),+)\rtimes_{\overline{\rho}}(B/\Ann(B),\circ)$.
For $c\in \Z_{\geq 1}$, let $m_1$ be the number of $H$-orbits of size $c$ and 
$m_2$ the $K$-orbits of size $c$.
Then 
\[
m_1=m_2|A|/|B|.
\]
\end{thm}

\begin{cor}
\label{cor:isoclinic orbits}
    Let $A$ and $B$ be finite isoclinic skew braces. 
Then 
\[
|B|\ell_A(m)=\ell_B(m)|A| \quad\text{ and }\quad|B|t_A(m)=t_B(m)|A|.
\]
\end{cor}
\begin{proof}
    The proof follows from Theorem~\ref{thm:isoclinc orbits} first for
    $H$ and $K$ to be the whole groups
    and then for $H=\{(-\bar{a},\bar{a})\colon a\in A\}$ and $K=\{(-\bar{b},\bar{b})\colon b\in B\}$.
\end{proof}

\begin{pro}
\label{pro:isoclinic}
    Let $A$ and $B$ finite skew braces of the same size. 
    If they are isoclinic,
    then $\Lambda(A)\cong\Lambda(B)$ and $\Theta(A)\cong\Theta(B)$.
\end{pro}

\begin{proof}
    By Corollary~\ref{cor:isoclinic orbits},
    since $A$ and $B$ are of the same size,
    they have the same number of $\lambda$-orbits 
    of size $m$ and the same number of $\theta$-orbits 
    of size $m$ for every $m\in\Z_{\geq1}$.
\end{proof}

The converse of Proposition~\ref{pro:isoclinic}
is not true in general, as we can find finite skew braces
of the same size and with the same graphs that are not isoclinic.
\begin{exa}
    Let $A$ be a finite abelian group and $G$ a non-abelian group of the same size as $A$.
    Then $\Lambda(\Triv(A))$ and $\Lambda(\Triv(G))$ have no vertices, 
    but $\Triv(A)$ and $\Triv(G)$ are not isoclinic, since their additive groups are not.
\end{exa}

The following example shows 
that the hypothesis on the size 
in Proposition~\ref{pro:isoclinic} is necessary.

\begin{exa}
    Let $A$ be the skew brace on $\Z/4\Z$ with $x\circ y=x+y+2xy$
    and $B$ be the skew brace on $\Z/8\Z$ with $x\circ y=x+5^xy$.
    Then $A$ and $B$ are isoclinic skew braces of abelian type.
    But $\Theta(A)=\Lambda(A)$ is the graph with one vertex,
    while $\Theta(B)=\Lambda(B)$ is the complete graph with two vertices.
\end{exa}

\begin{exa}
    Let $A$ be the skew brace on $\Z/9\Z$ with $x\circ y=x+y+3xy$
    and $B$ be the skew brace on $\Z/3\Z\times\Z/2\Z$
    with
    $$(n,m)\circ (s,t)=\left(n+(-1)^ms,m+t\right).$$
    Then $A$ and $B$ are not isoclinic but
    $\Theta(A)=\Lambda(A)\cong\Lambda(B)=\Theta(B)$
    is the complete graph with two vertices.
\end{exa}

However, even without assumption on the sizes,
some graph-theoretic properties are preserved under isoclinism.

\begin{pro}
Let $A$ and $B$ be isoclinic finite skew braces.
Then $\Lambda(A)$ is complete if and only if $\Lambda(B)$ is complete. Furthermore, the number of connected components of both graphs coincide.
The same holds for the graphs $\Theta(A)$ and $\Theta(B)$.
\end{pro}
\begin{proof}
    By Corollary~\ref{cor:isoclinic orbits}
    $\ell_A(m)\ne 0$ if and only if $\ell_A(m)\ne 0$,
    for every $m\in \Z_{\geq 0}$.
    Let $L_1$ and $L_2$
    be two vertices of $\Lambda(B)$.    
    Then there are two vertices $M_1$ and $M_2$ of $\Lambda(A)$
    such that $|M_1|=|L_1|$ and $|M_2|=|L_2|$.
    In particular, $L_1$ and $L_2$ are adjacent in $\Lambda(B)$ if and only if
    $M_1$ and $M_2$ are adjacent in $A$.
     With the same proof, we obtain the same result for $\Theta(A)$ and $\Theta(B)$.
    Hence the results follow.
\end{proof}

\section{Graphs of skew braces of small order}
\label{examples}
In this section, we characterize the common divisor graph for skew braces of order a product of two primes. 
\begin{pro}
Let $p$ be a prime number, $n\in\Z_{\geq 0}$, 
and $A$ be a skew brace of size $p^n$. 
Then $\Lambda(A)$ and $\Theta(A)$ are complete graphs on a number of vertices that is a multiple of $p-1$.
\end{pro}
\begin{proof}
    Firstly the cardinality of every non-trivial $\lambda$- or $\theta$-orbit divides $|A|=p^n$.
    So $\Lambda(A)$ and $\Theta(A)$ are complete graphs.
    Moreover, by Proposition~\ref{pro:p-groups orbits},
    $p-1$ divides $\ell(p^m)$ and $t(p^m)$ for all $m\in\Z_{\geq 0}$.
    So the total number of $\lambda$- and $\theta$-orbits is also divisible by $p-1$.
\end{proof}

Skew braces of size $p^2$, where $p$ is a prime number, 
were classified by Bachiller in~\cite[Proposition 2.4]{MR3320237}.

\begin{thm}[Bachiller]
\label{thm:p2}
 A complete list of the isomorphism classes of 
   skew braces of order $p^2$ is the following.
     \begin{enumerate}
        \item The trivial skew brace on $\Z/{p^2}\Z$.
        \item The trivial skew brace on $\Z/p\Z\times\Z/p\Z$.
        \item The skew brace on $\Z/{p^2}\Z$ where $x\circ y=x+y+pxy$.
        \item The skew brace on $\Z/p\Z\times\Z/p\Z$ where 
        \[
        (x_1,y_1)\circ(x_2,y_2)=(x_1+x_2+y_1y_2, y_1+y_2).
        \]
    \end{enumerate}
\end{thm}

\begin{pro}
\label{p2graph:pro}
Let $p$ be a prime number and $A$ be a skew brace of order $p^2$. 
Then $\Lambda(A)=\Theta(A)$. If $A$ is non-trivial, then $\Lambda(A)=\Theta(A)$ is the complete graph with $p-1$ vertices.
\end{pro}

\begin{proof}
Since $|A|=p^2$, it is of abelian type, so $\Lambda(A)=\Theta(A)$.
Moreover, the class equation for the $\lambda$-action implies that 
$\left|\Fix(A)\right|\ne1$. Hence $\left|\Fix(A)\right|=p$, as $A$ is non-trivial. 
Moreover, $\left|\Lambda(a)\right|=p$ for all $a\in A\setminus \Fix(A)$. 
By Proposition~\ref{pro:p-groups orbits}, the number of non-trivial $\lambda$-orbits of $A$ is $p-1$
and they all have size $p$.
\end{proof}

Skew braces of size $pq$, where $p$ and $q$ are different prime
numbers were classified 
by Acri and Bonatto in~\cite{MR4085764}. 

\begin{thm}
\label{thm:graph pq}
Let $p>q$ be two prime numbers such that $p\equiv1\pmod q$,
let $H$ be a group of order $pq$ and $G$ be a group of order dividing $p^2q^2$
that acts on $H$ by $\alpha:G\to\Aut(H)$.
\begin{enumerate}
\item If $\left|H^G\right|=pq$, then $\mathcal{C}(\alpha)$ has no vertices. 
\item If $\left|H^G\right|=p$, then $\mathcal{C}(\alpha)$ is the complete graph with $q-1$ vertices.
\item If $\left|H^G\right|=q$, then $\mathcal{C}(\alpha)$ is the complete graph with $p-1$ vertices.
\item If $\left|H^G\right|=1,$ then $\mathcal{C}(\alpha)$ has two connected components:
one is the complete graph with $q-1$ vertices and
the other is the complete graph with $\frac{p-1}{q}$ vertices.
\end{enumerate}
\end{thm}

\begin{proof}
Since $|G|$ divides $p^2q^2$, $|H|=pq$ and $p>q$,
the non-trivial orbits can be of size $p$, $q$ or $q^2$.
By the class equation,
\begin{equation}
\label{class equation pq:eq}
    pq=\left|H^G\right|+p\mathbf{c}(p)+q\mathbf{c}(q)+q^2\mathbf{c}(q^2).
\end{equation}
It follows that $\mathbf{c}(p)<q$, $\mathbf{c}(q)<p$, $\mathbf{c}(q^2)<\frac{p-1}{q}$ and $m=\mathbf{c}(q)+q\mathbf{c}(q^2)<p$.

If $\left|H^G\right|=pq$, then there are no non-trivial orbits and hence 
$\mathcal{C}(\alpha)$ has no vertices.

If $\left|H^G\right|=p$, Equation \eqref{class equation pq:eq} becomes $pq=p+p\mathbf{c}(p)+qm$. 
Thus $p\mid m$. Since $m<p$, then $m=0$.
and $\mathbf{c}(p)=q-1$.
Therefore $\mathcal{C}(\alpha)$ is the complete graph with $q-1$ vertices.

Similarly, if $\left|H^G\right|=q$, one proves 
that $\mathcal{C}(\alpha)$ is the complete graph
with $p-1$ vertices.
 
If $\left|H^G\right|=1$, then 
$pq=1+p\mathbf{c}(p)+qm$.
Thus $\mathbf{c}(p)\equiv p\mathbf{c}(p)\equiv -1 \pmod q$ and $m\equiv \frac{p-1}{q} \pmod p$.
Recalling that $\mathbf{c}(p)<q$ and $m<p$, it follows that $m=\frac{p-1}{q}$ and $\mathbf{c}(p)=q-1$.
Therefore $\mathcal{C}(\alpha)$ is a graph with two connected components:
one is the complete graph with $q-1$ vertices and
the other is the complete graph with $\frac{p-1}{q}$ vertices.
\end{proof}

\begin{cor}
\label{cor:lambda graph pq}
Let $p>q$ be two prime numbers such that $p\equiv1\pmod q$ 
and $A$ be a skew brace of size $pq$.
\begin{enumerate}
\item If $A$ is trivial, then $\Lambda(A)$ has no vertices. 
\item If $\left|\Fix(A)\right|=p$, then $\Lambda(A)$ is the complete graph with $q-1$ vertices.
\item If $\left|\Fix(A)\right|=q$, then $\Lambda(A)$ is the complete graph with $p-1$ vertices.
\item If $\left|\Fix(A)\right|=1,$ then $\Lambda(A)$ has two connected components:
one is the complete graph with $q-1$ vertices and
the other is the complete graph with $\frac{p-1}{q}$ vertices.
\end{enumerate}
\end{cor}

\begin{cor}
    \label{cor:theta graph pq}
Let $p>q$ be two prime numbers such that $p\equiv1\pmod q$ 
and $A$ be a skew brace of size $pq$.
\begin{enumerate}
\item If $A$ is trivial of abelian type, then $\Theta(A)$ has no vertices. 
\item If $\left|\Fix_\theta(A)\right|=p$, then $\Theta(A)$ is the complete graph with $q-1$ vertices.
\item If $\left|\Fix_\theta(A)\right|=q$, then $\Theta(A)$ is the complete graph with $p-1$ vertices.
\item If $\left|\Fix_\theta(A)\right|=1,$ then $\Theta(A)$ has two connected components:
one is the complete graph with $q-1$ vertices and
the other is the complete graph with $\frac{p-1}{q}$ vertices.
\end{enumerate}
\end{cor}

\begin{rem}
  If $A$ is a skew brace of size $pq$, then 
    $\left|\Fix(A)\right|$ and $\left|\Fix_\theta(A)\right|$
    are numbers in $\{1,p,q,pq\}$. 
    Thus Corollaries~\ref{cor:lambda graph pq} and~\ref{cor:theta graph pq}
    describe the graphs of all possible skew braces of size $pq$.  
\end{rem}

\begin{lem}\label{lem: Fix theta size pq}
Let $p>q$ be two prime numbers
and $A$ be a skew brace of size $pq$.
Then if $A$ is of abelian type, $\Fix_\theta(A)=\Fix(A)$. Otherwise $\Fix_\theta(A)=\{0\}$.
\end{lem}
\begin{proof}
    This follows easily from the fact that the centre of a non-abelian group of order $pq$ is trivial.
\end{proof}

We now recall the classification of skew braces of order $pq$ in the main theorem of~\cite{MR4085764}. 
Let $p$ and $q$ be prime numbers such that $p\not\equiv 1 \pmod q$. 
Then there is only one skew brace of size $pq$, the trivial one. 

\begin{thm}[Acri--Bonatto]
\label{pq:thm}
   Let $p>q$ be primes such that $p\equiv 1 \pmod q$ and 
   $g$ be a fixed element of $(\Z/p\Z)^\times$ of order $q$. 
   A complete list of the $2q+2$ isomorphism classes of 
   skew braces of order $pq$ is the following.
   \begin{enumerate}
   \item The trivial skew brace on $\Z/p\Z\times\Z/q\Z$.
   \item The trivial skew brace on $\Z/p\Z\rtimes_g\Z/q\Z$. 
   \item The additive group $\Z/p\Z\times\Z/q\Z$ with $(n,m)\circ(s,t)=(n+g^ms,m+t)$
   \item Additive group $\Z/p\Z\rtimes_g\Z/q\Z$ and $(n,m)\circ(s,t)=(f(n,m,s,t),m+t)$
        where
        \begin{enumerate}
            \item $f(n,m,s,t)=g^tn+g^ms$. 
            \item $f(n,m,s,t)=n+(g^\gamma)^ms$ for all $1<\gamma\leq q$.
            \item $f(n,m,s,t)=g^tn+(g^\mu)^ms$ for all $1<\mu\leq q$. 
        \end{enumerate}
   \end{enumerate}
\end{thm}

\begin{pro}
    Let $p$ and $q$ be different prime numbers and 
    $A$ and $B$ be a skew brace of size $pq$. Then $A$ and $B$ are 
    isoclinic if and only if $A$ and $B$ are isomorphic. 
\end{pro}

\begin{proof}
    By Theorem~\ref{pq:thm} and since $Z(\Z/p\Z\rtimes_g\Z/q)=\{0\}$,
    then $\Ann(A)=\{0\}$ unless $A$ is the trivial skew brace on $\Z/p\Z\times\Z/q\Z$.
    The same holds for $B$.
    Thus $A$ and $B$ are isoclinic if and only if they are isomorphic.
\end{proof}

Now we can go through the classification 
of skew braces of size $pq$ and compute their graphs. 
For that purpose, we introduce the following notation
for the skew braces of Theorem~\ref{pq:thm} (Table~\ref{tab:pq}).

\begin{table}[H]
\caption{Skew braces of size $pq$.}
\label{tab:pq}
\begin{center}
\begin{tabular}{|c|c|c|} 
 \hline
 Notation & Theorem & Comments\\ 
 \hline
 $T_1$ & \ref{pq:thm}(1) & Trivial\\ 
 $T_2$ & \ref{pq:thm}(2) & Trivial\\ 
 $C$ & \ref{pq:thm}(3) & Bi-skew\\ 
 $D$ & \ref{pq:thm}(4a)& Two-sided\\ 
 $E_\gamma$ & \ref{pq:thm}(4b) & $1<\gamma\leq q$, Bi-skew \\ 
 $F_\mu$ & \ref{pq:thm}(4c) & $1<\mu\leq q$\\ 
 \hline
\end{tabular}
\end{center}
\end{table}

\begin{thm}
\label{pq:graph:thm}
Using the notation just introduced, the following hold.
    \begin{itemize}
        \item $\Lambda(T_1 )$, $\Lambda(T_2)$ and $\Theta(T_1 )$ have no vertices.
        \item $\Lambda(C )$, $\Lambda(E_\gamma )$ and $\Theta(C )$ are complete graphs with $p-1$ vertices.
        \item $\Lambda(D )$ is the complete graph with $q-1$ vertices.
        \item $\Lambda(F_\mu )$, $\Theta(T_2 )$, $\Theta(D )$, $\Theta(E_\gamma )$ and $\Theta(F_\mu )$ for all $1<\gamma\leq q$ and $1<\mu\leq q$ are the graph with two connected components: one is the complete graph with $q-1$ vertices and the other is the complete graph with $\frac{p-1}{q}$ vertices.
    \end{itemize}
\end{thm}
\begin{proof}
Let's start computing the $\Lambda$-graphs.
Clearly, trivial skew braces have a graph without vertices.
Let $A$ be a non-trivial skew brace of order $pq$. Then, by Theorem~\ref{thm:graph pq}, we just need to compute the order of $\Fix(A)$. 

If $A=C$, and $(n,m),(s,t)\in C$.
 \[
\lambda_{(n,m)}(s,t)=(-n,-m)+(n+g^ms,m+t)=(g^ms,t).
\] 
Thus $(s,t)\in\Fix(C)$ if and only if $g^ms=s$ for all $m\in\Z/q\Z$, which means that $s=0$. Hence $\Fix(C)=\{0\}\times\Z/q\Z$ has order $q$.
    
In all the other cases, 
\[
\lambda_{(n,m)}(s,t)=\left(-g^{-m}n+g^{-m}f(n,m,s,t),t\right)=\left(g^{-m}\left(-n+f(n,m,s,t)\right),t\right).
\]

Thus $(s,t)\in\Fix(A)$ if and only if $-n+f(n,m,s,t)=g^ms$ in $\Z/p\Z$, for all $n\in\Z/p\Z$ and $m\in\Z/q\Z$.

In the case $A=D$ we have that $f: (n,m,s,t)\mapsto g^tn+g^ms$, 
so $(s,t)\in\Fix(D)$ if and only if $-n+g^tn+g^ms=g^ms$ for all $n\in\Z/p\Z$ and $m\in\Z/q\Z$.
Therefore $\Fix(D)=\Z/p\Z\times\{0\}$ has order $p$.

If $A=E_\gamma$, we have $f: (n,m,s,t)\mapsto n+(g^\gamma)^ms$,
so $(s,t)\in\Fix(E_\gamma)$ if and only if 
$-n+n+(g^\gamma)^ms=g^ms$ for all $n\in\Z/p\Z$ and $m\in\Z/q\Z$.
Hence $\Fix(E_\gamma)=\{0\}\times\Z/q\Z$ has order $q$.

Finally,  $f: (n,m,s,t)\mapsto g^tn+(g^\mu)^ms$ for $A=F_\mu$,
so $(s,t)\in\Fix(F_\mu)$ if and only if
$-n+g^tn+(g^\mu)^ms=g^ms$ for all $n\in\Z/p\Z$ and $m\in\Z/q\Z$.
In particular, $g^t=1$ and $g^\mu s=gs$.
Thus $\Fix(F_\mu)=\{0\}\times\{0\}$ has order 1.

In order to compute the $\theta$-graph, 
thanks to Theorem~\ref{cor:theta graph pq} and 
Lemma~\ref{lem: Fix theta size pq}, we 
just need to check which skew braces are of abelian type.
Hence the results follow.
\end{proof}

\section{Skew braces with graphs with two disconnected vertices}\label{sec:2vertices}
In this section, we study skew braces whose graphs have exactly two disconnected vertices. Note that in~\cite{MR1099007} it is proved that for groups this is possible only for
the symmetric group $\S_3$.
So we have the same $\lambda$-graph for the almost trivial skew brace on $\S_3$, and actually, this is the only case.
The situation for the $\theta$-graph is more involved, but we show that this only appears in skew braces of order six (Table~\ref{SB6}).

\begin{table}[H]
\caption{Skew braces of size $6$, where $f=\left|\Fix(A)\right|$.}
\label{SB6}
\begin{center}
\begin{tabular}{|c|c|c|c|c|c|} 
 \hline
 $A$ & $(A,+)$ & $(n,m)\circ(s,t)$ & $f$ & $\Lambda(A)$ & $\Theta(A)$\\ 
 \hline
$T_1 $ & $\Z/3\Z\times\Z/2\Z$ & $(n+s,m+t)$ 
& $6$ & & \\ 
$T_2 $ & $\Z/3\Z\rtimes_{-1}\Z/2\Z$ & $(n+(-1)^ms,m+t)$
& $6$ & & $\bullet\quad\bullet$ \\ 
$D $ & $\Z/3\Z\rtimes_{-1}\Z/2\Z$ & $((-1)^tn+(-1)^ms,m+t)$
& $3$ & $\bullet$ & $\bullet\quad\bullet$\\ 
$C $ & $\Z/3\Z\times\Z/2\Z$ & $(n+(-1)^ms,m+t)$
& $2$ & $\bullet$---$\bullet$ & $\bullet$---$\bullet$\\ 
$E_2 $ & $\Z/3\Z\rtimes_{-1}\Z/2\Z$ & $(n+s,m+t)$
& $2$ & $\bullet$---$\bullet$ & $\bullet\quad\bullet$\\  
$F_2 $ & $\Z/3\Z\rtimes_{-1}\Z/2\Z$ &$((-1)^tn+s,m+t)$ 
& $1$ & $\bullet\quad\bullet$ & $\bullet\quad\bullet$\\ 
 \hline
\end{tabular}
\end{center}
\end{table}

\begin{lem}
\label{2discvertices:fix:lem}
Let $G$ and $H$ be finite groups 
and let $\alpha:G\to\Aut(H)$ be an action.
If $\mathcal{C}(\alpha)$ has exactly two disconnected vertices. 
Then $H^G=\{0\}$.
\end{lem}

\begin{proof}
    Since $\mathcal{C}(\alpha)$ has two disconnected vertices,
    there are only two non-trivial orbits of coprime orders $m_1$ and $m_2$.
    By Proposition~\ref{fix:pro}, $\left|H^G\right|$ divides both $m_1$ and $m_2$, so $\left|H^G\right|=1$.
\end{proof}

\begin{pro}
    Let $A$ be a finite skew brace. If $\Lambda(A)$ has exactly two disconnected vertices, then $A$ is the almost trivial skew brace on $\S_3$. 
\end{pro}
\begin{proof}
    Let $\Lambda(a)$ and $\Lambda(b)$ be the two different non-trivial orbits and 
    $m_1=\left|\Lambda(a)\right|$ and $m_2=|\Lambda(b)|$. By Lemma~\ref{2discvertices:fix:lem}, $\left|\Fix(A)\right|=1$. Thus the class equation in this case becomes
    $|A|=1+m_1+m_2$. Moreover, since $m_1$ and $m_2$ are coprime, 
    by Proposition~\ref{pro: coprime orbits}, we also have that $\Stab(a)\circ\Stab(b)=A$. Thus there exists a positive integer $h$ such that 
    \[
    h|A|=\left|\Stab(a)\right|\left|\Stab(b)\right|=\frac{|A|}{m_1}\frac{|A|}{m_2}.
    \]
    Hence $|A|=hm_1m_2$ and so 
    $1+m_1+m_2=hm_1m_2$.
    Suppose now, without loss of generality, that $1<m_1<m_2$. 
    Then 
    \[
    2m_2\geq 1+m_1+m_2=hm_1m_2.
    \]
    It follows that $m_1=2$ and $h=1$ and hence $|A|=6$. 
    By Theorem~\ref{thm:graph pq}, 
    $A$ is isomorphic to 
    $F_\mu$ on $\Z/3\Z\rtimes_{-1}\Z/2\Z$ with $\mu=2$. Thus 
    \[
    (n,m)\circ (s,t)=\left((-1)^tn+s,m+t\right)=(s,t)+(n,m).
    \]
    It is easy to check that this is actually isomorphic to the almost trivial skew brace on $\S_3$.
\end{proof}

\begin{lem}[Catalan's conjecture\cite{MR2076124}]
\label{lem: Catalan conj}
The only solution in the natural numbers of $x^a-y^b=1$, for $a,b > 1$ and  $x,y>0$ is $x=3, a = 2, y = 2, b = 3$.    
\end{lem}

\begin{pro}
    Let $A$ be a finite skew brace. If $\Theta(A)$ has exactly two disconnected vertices, then $A$ is of size 6 and is one of these possible skew braces with notation as in Table~\ref{tab:pq}:
    \begin{itemize}
        \item the trivial skew brace on  $\S_3$ (the skew brace $T_1 $ of order 6);
        \item the bi-skew brace $D$ of order 6;
        \item the bi-skew brace $E_2 $ of order 6;
        \item the almost trivial skew brace on $\S_3$ (the skew brace $F_2 $ of order 6).
    \end{itemize}
\end{pro}
\begin{proof}
    Let $\Theta(a)$ and $\Theta(b)$ be the two non-trivial $\theta$-orbits 
    with $m_1=\left|\Theta(a)\right|$ and $m_2=|\Theta(b)|$.
    By Lemma~\ref{2discvertices:fix:lem}, $\left|\Fix_\theta(A)\right|=1$.
    Moreover there exist two different prime numbers $p$ and $q$ such that 
    $p$ divides $m_1$ and $q$ divides $m_2$.
    Since the elements of a $\theta$-orbits have all the same additive order
    and $\Fix_\theta(A)$ is trivial,
    $\Theta(a)$ consists of all the elements of additive order $q$ and 
    $\Theta(b)$ consists of all the elements of additive order $p$.
    Hence $|A|=p^\alpha q^\beta$ for some $\alpha,\beta\in\Z_{>0}$ and
    $m_1=p^k$ and $m_2=q^t$ for some $1\leq k\leq \alpha$ and $1\leq t\leq \beta$.
    Thus the class equation in this case becomes
    \[
    p^\alpha q^\beta=|A|=1+m_1+m_2=1+p^k+q^t.
    \]
    Let $F$ be the Fitting subgroup of $(A,+)$. Since $F$ is a characteristic subgroup, it is a union of $\theta$-orbits.
    Since $|A|=p^\alpha q^\beta$, $(A,+)$ is solvable and so $F\neq\{0\}$.
    Moreover $F\neq A$, otherwise $(A,+)$ would be nilpotent with no elements of order $pq$,
    which is not possible.
    Therefore, without loss of generality, we can assume that $F=\{0\}\cup \Theta(b)$.
    Thus $F$ is the unique Sylow $p$-subgroup of $(A,+)$, so $|F|=p^\alpha$.
    Hence $1+q^t=p^\alpha$.
    By Catalan's conjecture (Lemma~\ref{lem: Catalan conj}),
    either $p=3, \alpha=2, q=2$ and $t=3$ or
    $\alpha\leq 1$ or $t\leq 1$.
    
    If $p=3, \alpha=2, q=2$ and $t=3$, then the class equation becomes
        \[
        3^2 2^\beta=1+3^k+2^3.
        \]
        But $k\leq 2$, so $1+3^k+2^3\leq 18$ and $3\leq \beta$, so $3^2 2^\beta=9\cdot 8\cdot 2^{\beta-3}\geq 72$, a contradiction.
        
    If $\alpha\leq 1$, then $k=\alpha=1$. Therefore $p=1+q^t$ and so $|A|=1+p+q^t=2p$.
        Thus $q=2$ and $\beta=1$, so $|A|=2p$. 
        Hence, by Theorem~\ref{pq:graph:thm}, $p=3$ and the only possibilities in the skew braces of order 6 (Table~\ref{SB6}) are
        the four listed in the statement.
        
    If $t\leq 1$, then $t=1$ and $p^\alpha=1+q$.
        So either $q=2, p=3$ and $k=\alpha=1$ or $q\neq 2$ and $p=2$.
        In the first case, we obtain that $|A|=1+3+2=6$ and by Theorem~\ref{pq:graph:thm}, the only possibilities are
        the four listed in the statement.  
        In the second case, $2^\alpha=1+q$ and the class equation becomes
        \[
        2^\alpha q^\beta=1+2^k+q=2^k+2^\alpha=2^k(1+2^{\alpha-k}).
        \]
        Hence $k=\alpha$ and $q^\beta=1+2^{\alpha-k}=2$, a contradiction with $q\neq 2$.  
\end{proof}

\section{Skew Braces with one vertex \texorpdfstring{$\lambda$}{lambda}-graph: Generalities}\label{sec:onevertex}
In the following sections, we study skew braces whose graph has exactly one vertex. Note that this can not happen in the non-commuting graph of a group, as this implies that the factor by the center of the group is cyclic, which implies that the group is abelian, so has a graph without vertices. The characterization will be divided over three sections. The first sets up a general framework. In the next section, we treat skew braces of abelian type. Afterwards, we deal with general skew braces.

\begin{lem}
\label{1vertex:leftnilp:lem}
Let $A$ be a finite skew brace. If $\Lambda(A)$ has exactly one vertex, 
then $A$ is left nilpotent of class two, $\Fix(A)$ has index two 
and $A^2=\Fix(A)$.
\end{lem}

\begin{proof}
As $(A,\circ)$ acts on $(A,+)$ via the $\lambda$-action, we obtain, by the class equation, that 
\[
|A| = \left|\Fix(A)\right| + m, 
\]
 where $m$ denotes the number of elements in the unique non-trivial 
 vertex of the graph $\Lambda(A)$. Denote $f=\left|\Fix(A)\right|$ and $k=\left\lbrack A \colon \Fix(A) \right\rbrack$. Then $|A|=kf$. Hence there exists a positive integer $l$ such that $m=lf$. Using the class equation, $k=1+l$.
Since $m$ divides $|A|$, $l$ divides $k$. As this implies that $l$ divides $l+1$, 
this shows that $l=1$. Hence $m=f$ and $|A| = 2\left|\Fix(A)\right|$.
As the index of $\Fix(A)$ in $A$ is two, it follows that this is both an additive and multiplicative normal subgroup, so $\Fix(A)$ is an ideal and $A/\Fix(A) \cong \Triv(\mathbb{Z}/2\Z)$. In particular, $A^{2}\subseteq \Fix(A)$, 
which implies that $A^{3} \subseteq A*\Fix(A)=0$. This shows that $A$ is a left nilpotent brace. 
Moreover, $A^2\subseteq \Fix(A)$ and, since $A$ has a unique non trivial orbit of index two, 
by Proposition~\ref{sizeA2:pro}, $|A^2|\geq |A|/2=\left|\Fix(A)\right|$. Thus $A^2=\Fix(A)$.
\end{proof}

The following result is a direct consequence of
Lemma~\ref{1vertex:leftnilp:lem} and~\cite[Lemma 4.5]{MR3957824}.

\begin{pro}
Let $A$ be a finite skew brace. If $\Lambda(A)$ has exactly one vertex, then the additive group of $A^2$ is abelian and $A^2$ is a trivial skew brace.
\end{pro}

\begin{lem}
\label{1vertex:A/ker=Fix:lem}
Let $A$ be a finite skew brace. If $\Lambda(A)$ has exactly one vertex, 
then $(A/\ker \lambda,\circ)$ and $(\Fix(A),+)$ are isomorphic abelian groups.
Moreover, 
\[
x+f-x=-f
\]
for all $f\in\Fix(A)$ and $x\in A\setminus\Fix(A)$.
\end{lem}
\begin{proof}
  Let $x\in A\setminus\Fix(A)$. Then $\Lambda(x)$ is the unique non-trivial $\lambda$-orbit of $A$.
  By Lemma~\ref{1vertex:leftnilp:lem}, $\left|\Fix(A)\right|=\left|\Lambda(x)\right|=|A|/2$ and so $\Lambda(x)=x+\Fix(A)$.
   Thus, denoting 
   \[
   \Aut_{\Fix(A)}(A,+)=\big\{\varphi\in\Aut(A,+)\colon \varphi|_{\Fix(A)}=\id\big\},
   \]
   where 
   $\varphi|_{\Fix(A)}$ is the restriction 
   of $\varphi$ on $\Fix(A)$, 
   we have a well-defined map 
   \[
       h\colon \Aut_{\Fix(A)}(A,+)\to \Fix(A),\quad 
       \varphi\mapsto -x+\varphi(x). 
   \]
   Indeed, if $\varphi|_{\Fix(A)}=\id$
   then $\varphi(x)\in\Lambda(x)=x+\Fix(A)$
   and $-x+\varphi(x)\in\Fix(A)$.
    Moreover, for all $\varphi,\psi\in \Aut_{\Fix(A)}(A,+)$, 
   \[
   h(\varphi\psi)=-x+\varphi(\psi(x))=-x+\varphi(x-x+\psi(x))=-x+\varphi(x)+\varphi(-x+\psi(x)),
   \]
   and $-x+\psi(x)\in\Fix(A)$ implies  that $\varphi(-x+\psi(x))=-x+\psi(x)$. Thus
   \[
   h(\varphi\psi)=-x+\varphi(x)-x+\psi(x)=h(\varphi)+h(\psi).
   \]
   and $h$ is a group homomorphism. It is also injective:
   $h(\varphi)=h(\psi)$ if and only if $\varphi(x)=\psi(x)$ and,
   since $\varphi|_{\Fix(A)}=\id=\psi|_{\Fix(A)}$,
   then $\varphi=\psi$.

  On the other hand, $\lambda$ induces an isomorphism between
  $(A/\ker \lambda,\circ)$  and a subgroup of $\Aut_{\Fix(A)}(A,+)$.
  Therefore $(A/\ker\lambda,\circ)$ is isomorphic to a subgroup of $\Fix(A)$.
  Since $\left|\ker\lambda\right|\leq 2,$ then $\left|\Fix(A)\right|=|A|/2$ divides $|A|/\left|\ker\lambda\right|$.
  Thus we obtain an isomorphism 
  \begin{align*}
      H\colon (A/\ker\lambda,\circ)\to (\Fix(A),+),\quad 
      \overline{a}\mapsto -x+\lambda_a(x).
  \end{align*}
  Finally, $x+x\in\Fix(A),$ as $\Fix(A)$ has index two in $A$.  Hence, for every element $f\in \Fix(A)$, there is a unique $\overline{a}\in A/\ker\lambda$ such that $f=H(\overline{a})$ and
  \[
x+x=\lambda_a(x+x)=\lambda_a(x)+\lambda_a(x)=x+H(\overline{a})+x+H(\overline{a})=x+f+x+f,
  \]
  hence $x+f-x=-f$.
\end{proof}

The following result easily follows, where it is used that every normal subgroup of order 2 is central.

\begin{cor}
\label{1vertex:kercentral:cor}
     Let $A$ be a finite skew brace. If $\Lambda(A)$ has exactly one vertex, then
     $\ker\lambda$
     has order 2 and $(A/\ker\lambda,\circ)$ is abelian. 
     In particular, $\ker\lambda\subseteq Z(A,\circ)$.
\end{cor}

\begin{thm}
\label{1vertex:2odd:thm}
    Let $A$ be a finite skew brace such that $\Lambda(A)$ has exactly one vertex.
    If $\ker\lambda\cap\Fix(A)=\{0\}$,
    then $A$ is isomorphic to the skew brace 
    on $\Fix(A)\times \Z/2\Z$ with operations
    \begin{equation}
    \label{eq:operations}
    \begin{aligned}
    (f_1,k_1)+(f_2,k_2)&=\big(f_1+(-1)^{k_1}f_2,k_1+k_2\big),\\
    (f_1,k_1)\circ(f_2,k_2)&=\big((-1)^{k_2}f_1+(-1)^{k_1}f_2,k_1+k_2\big).
    \end{aligned}
    \end{equation}
    Moreover, $\Fix(A)$ is a trivial skew brace of abelian type of odd order. 
    
    Conversely, 
    if $F$ is a non-trivial abelian group of odd order, then the set $A=F\times \Z/2\Z$ with operations
    \eqref{eq:operations} is a skew brace such that 
    $\Fix(A)=F\times\{0\}$, $\ker\lambda=\{0\}\times\Z/2\Z$ and 
    $\Lambda(A)$ has exactly one vertex.
\end{thm}

\begin{proof}
Since $\ker\lambda\cap\Fix(A)=\{0\}$, the only non-trivial orbit of $A$ is $\Lambda(x)$ 
for any $x\in\ker\lambda\setminus\{0\}$.
By Lemma~\ref{1vertex:leftnilp:lem}, $\Fix(A)$ is an ideal of $A$.

Since $\left|\ker\lambda\right|\left|\Fix(A)\right|=|A|$ and $\left|\ker\lambda\right|=2$, 
we have a direct decomposition 
 \[
 (A,\circ)=\Fix(A)\circ\ker\lambda\cong (\Fix(A),+)\times\Z/2\Z
 \]
 of $(A,\circ)$,  
 and a semidirect decomposition 
 \[
 (A,+)=\Fix(A)+\ker\lambda\cong (\Fix(A),+)\rtimes\Z/2\Z
 \]
 of $(A,+)$ with the property that $x+f-x=-f$ for all $f\in\Fix(A)$, as shown in 
 Lemma ~\ref{1vertex:A/ker=Fix:lem}.
 For every $a,b\in A$ there are unique $(f_a,k_a),(f_b,k_b)\in \Fix(A)\times \Z/2\Z$ 
 such that $a=f_a+k_ax,b=f_b+k_bx$; their sum is
 \begin{equation}
 \label{additivelaw:eq}
     a+b=f_a+k_ax+f_b+k_bx=f_a+(-1)^{k_a}f_b+(k_a+k_b)x.
 \end{equation}

 By skew left distributivity and since $f_b\in\Fix(A)$ and, by Corollary~\ref{1vertex:kercentral:cor}, $\ker\lambda\subseteq Z(A,\circ)$
,
     \begin{align*}
         a\circ b 
         &=(f_a+k_ax)\circ(f_b+k_bx)\\
         &=(f_a+k_ax)\circ f_b -(f_a+k_ax)+ (f_a+k_ax)\circ k_bx\\
         &=f_a+k_ax+f_b-k_ax-f_a+k_bx\circ(f_a+k_ax).
     \end{align*}
By Equation \eqref{additivelaw:eq} and skew left distributivity,
this is further equal to
\[
f_a+(-1)^{k_a}f_b+k_ax-k_ax-f_a+k_bx\circ f_a-k_bx+k_bx\circ k_ax.
\]
In turn, as $k_bx\in\ker\lambda$ and $f_a\in \Fix(A)$, this is equal to
\[
(-1)^{k_a}f_b+k_bx+f_a+k_ax=(-1)^{k_a}f_b+(-1)^{k_b}f_a+(k_a+k_b)x.
\]
 Hence the isomorphism follows.
 
 Moreover, $\left|\Fix(A)\right|$ is odd, otherwise there would be an $f\in\Fix(A)\setminus\{0\}$ of order two. Thus, for every $b\in A$
 \[
 \lambda_f(b)=-f+f\circ(f_b+k_bx)=-f+f_b+(-1)^kf_b+k_bx=f_b+k_bx=b,
 \]
 a contradiction, as $f\in \Fix(A)\cap\ker\lambda=\{0\}$.
 
Let now $(F,+)$ be an abelian group of odd order. The $\lambda$-action of $A$ 
is given by
     \begin{align*}
         \lambda_{(f_1,k_1)}(f_2,k_2)&=-(f_1,k_1)+(f_1,k_1)\circ(f_2,k_2)\\
         &=\big((-1)^{k_1+1}f_1,k_1\big)+\big((-1)^{k_2}f_1+(-1)^{k_1}f_2,k_1+k_2\big)\\
         &=\big((-1)^{k_1+1}f_1+(-1)^{k_1+k_2}f_1+(-1)^{k_1+k_1}f_2,k_2\big)\\
         &=\big((-1)^{k_1}(-1+(-1)^{k_2})f_1+f_2,k_2\big).         
     \end{align*}
Then $(f_2,k_2)\in\Fix(A)$ if and only if $(-1)^{k_1}(-1+(-1)^{k_2})f_1=0$ in $F$ for all $(f_1,k_1)\in F\times\Z/2\Z$.  
Since $F$ has odd order, we conclude that $\Fix(A)=F\times\{0\}$.
Moreover, $(f_1,k_1)\in\ker\lambda$ if and only if $(-1)^{k_1}(-1+(-1)^{k_2})f_1=0$ in $F$ for all $(f_2,k_2)\in F\times\Z/2\Z$.
Hence, since $|F|$ is odd, $\ker\lambda=\{0\}\times\Z/2\Z$.
Finally
\[
\Lambda(0,1)=\left\{\big(-2(-1)^{k}f,1\big)\colon (f,k)\in F\times\Z/2\Z\right\}=F\times\{1\}
\]
and thus $\Lambda(A)$ has exactly one vertex.
\end{proof}

\begin{rem}
\label{rem:1vertex:2odd:iso}
    Given a non-trivial abelian group $F$ of odd order $d$,
    we will denote the skew brace constructed in this theorem by $D_{2d}(F)$.
    Note that two non-trivial abelian groups $F$ and $F'$
    of the same odd order $d$ are isomorphic if and only if
    $D_{2d}(F)$ and $D_{2d}(F')$ are isomorphic skew braces.
\end{rem}

We end the section with a theorem that allows inductive reasoning.

\begin{thm}
\label{1vertex:2even:thm}
    Let $A$ be a finite skew brace such that $\Lambda(A)$ has exactly one vertex.
    If $\ker\lambda\subseteq\Fix(A)$, then 
    $|A|$ is divisible by $4$. Moreover, 
    either $|A|=4$ 
    or $A/\Soc(A)$ is a skew brace with abelian multiplicative group and $\Lambda(A/\Soc(A))$ has exactly one 
    vertex. 
 \end{thm}
 
\begin{proof}
Since $\ker \lambda\subseteq\Fix(A)$ and, by Corollary~\ref{1vertex:kercentral:cor} $\ker\lambda\subseteq Z(A,\circ)$, 
\[
k+a=k\circ a=a\circ k=a+k,
\]
for all $k\in\ker\lambda$ and $a\in A$. 
Hence $\Soc(A)=\ker\lambda\subseteq\Fix(A)$.
Let $\Lambda(x)$ be the only non-trivial $\lambda$-orbit of $A$ and
$\overline{A}=A/\Soc(A)=A/\ker\lambda$. Then 
$\overline{A}$ has order $|A|/2$ and $\left\lbrack\overline{A}\colon\Fix(\overline{A})\right\rbrack\leq2$.

If $\left\lbrack\overline{A}\colon\Fix(\overline{A})\right\rbrack=1$,
then $\overline{A}$ is a trivial skew brace. 
By Lemma ~\ref{1vertex:A/ker=Fix:lem}, we have a group isomorphism $H\colon(\overline{A},\circ)\to (\Fix(A),+)$
defined by $H(\overline{a})=-x+\lambda_a(x)$.
But $\overline{A}$ is trivial, thus $H(\overline{a})=-x+\lambda_a(x)\in \Soc(A)$. 
Hence $\Fix(A)=\Soc(A)$ and so $|A|/2=\left|\Fix(A)\right|=\left|\Soc(A)\right|=2$. Therefore $|A|=4$.

If $\left\lbrack\overline{A}\colon\Fix(\overline{A})\right\rbrack=2$,
then $\left|\Fix(\overline{A})\right|=|\overline{A}|/2$. Hence 4 divides $|A|$. 
Moreover $\Lambda(\overline{A})$ has only one vertex because 
$\Lambda_{\overline{A}}(\overline{x})\supseteq \overline{\Lambda(x)}$.
\end{proof}

\section{Skew braces with one-vertex \texorpdfstring{$\lambda$}{lambda}-graphs: Abelian type}

\begin{lem}
\label{1vertex:nodecomp:lem}
Let $B$ be a finite skew brace of abelian type. If $\Lambda(B)$ has exactly one vertex, then there exists no decomposition $B=B_1\times B_2$ with $B_1$ and $B_2$ non-zero skew braces.
\end{lem}
\begin{proof}
Assume such a decomposition $B=B_1\times B_2$ does exist. Since $\Lambda(B)$ has one vertex, 
$B_1$  or $B_2$ is a non-trivial skew brace. Without loss of generality, we may assume $B_2$ is non-trivial. By Lemma 
~\ref{1vertex:leftnilp:lem}, $B$ is left nilpotent. It follows that also $B_1$ is left nilpotent. Hence there exist $a,b \in \Fix(B_1)$ with $a\ne b$. Let $c \in B_2$ be such that $c \in B\setminus \Fix(B_2)$. Then 
\[
\Lambda_B(a,c)= \left\lbrace a\right\rbrace \times \Lambda_{B_2}(c) \neq \left\lbrace b\right\rbrace \times \Lambda_{B_2}(c) = \Lambda_B(b,c).
\]
Thus $\Lambda(B)$ has two or more vertices, a contradiction.
\end{proof}

\begin{pro}
\label{1vertex:2brace:pro}
Let $B$ be a finite skew brace of abelian type. If $\Lambda(B)$ has exactly one vertex, then $B$ has size $2^m$ for some $m$. 
\end{pro}

\begin{proof}
By Lemma ~\ref{1vertex:leftnilp:lem}, $B$ is left nilpotent with $B^{3}=0$. By~\cite[Theorem 4.8, Corollary 4.3]{MR3957824}, 
$B$ is the direct product $B=B_1\times\cdots\times B_n$ of the $p$-skew braces corresponding to the Sylow subgroups of $(B,+)$. However, Lemma~\ref{1vertex:nodecomp:lem} dictates that such a decomposition can only have one non-zero factor. Hence, $B$ is a $p$-skew brace for some prime $p$. As, by Lemma~\ref{1vertex:leftnilp:lem}, $|B|=2\left|\Fix(B)\right|$, it follows that $B$ is a $2$-skew brace.
\end{proof}

\begin{pro}
\label{pro:1vertex quotient}
    Let $B$ be a finite skew brace of abelian type.
    If $\Lambda(B)$ has exactly one vertex, then $(B/\Soc(B),\circ)\cong(\Fix(B),+)$ is an elementary abelian 2-group.
    Moreover, if $|B|>4$, then $\Lambda(B/Soc(B))$ has exactly one vertex too.
\end{pro}
\begin{proof}
    Let $\Lambda(x)$ be the only non-trivial $\lambda$-orbit of $B$.
    By Lemma~\ref{1vertex:A/ker=Fix:lem}, one obtains $x+f-x=-f$ for all $f\in\Fix(B)$.
    Since $(B,+)$ is abelian, $f+f=0$ for all $f\in\Fix(B)$.
    Thus $\Fix(B)$ is an elementary abelian 2-group.
    
    Since $B$ is of abelian type and by Corollary~\ref{1vertex:kercentral:cor},
    $\ker\lambda\subseteq Z(B,\circ)$,
  it follows that $\ker\lambda\subseteq \Fix(B)$. 
  By Theorem~\ref{1vertex:2even:thm} and using that $|B|>4$,
  the graph $\Lambda(B/\Soc(B))$ has exactly one vertex.    
\end{proof}

\begin{lem}
\label{1vertexabelian:oreder4:lem}
Let $B$ be a finite skew brace of abelian type with abelian multiplicative group.
If $\Lambda(B)$ has exactly one vertex, then $|B|=4$.
\end{lem}

\begin{proof}
Since $(B,\circ)$ is abelian, $\Soc(B)=\Fix(B)$.
By Lemma~\ref{1vertex:leftnilp:lem}, it follows that $2\left|\Fix(B)\right|=|B|$ and $\left|\Soc(B)\right|=\left|\ker \lambda\right|=2$. Thus $|B|=4$. 
\end{proof}

\begin{pro}
\label{1vertex:ordermax8:pro}
Let $B$ be a finite skew brace of abelian type. If $\Lambda(B)$ has exactly one vertex, then $B$ is of size at most $8$.
\end{pro}

\begin{proof}
By Proposition~\ref{1vertex:2brace:pro}, $|B|=2^m$ for some $m$. If $m\geq 4$, then 
Lemma~\ref{1vertexabelian:oreder4:lem} implies that 
$(B,\circ)$ is non-abelian. Moreover, $\overline{B}=B/\Soc(B)$ has at least eight elements 
and $\Lambda(\overline{B})$ has only one vertex (Proposition~\ref{pro:1vertex quotient}).
Thus we obtain a contradiction by applying Lemma~\ref{1vertexabelian:oreder4:lem} and
Proposition~\ref{pro:1vertex quotient}
to $\overline{B}$, since $(\overline{B},\circ)$ would   
be at the same time non-abelian and elementary abelian.
\end{proof}

\begin{thm}
\label{thm: one vertex lambda of abelian type}
    Let $B$ be a finite skew brace of abelian type such that $\Lambda(B)$ has only one vertex. Then $B$ is isomorphic to one of the following skew braces.
    \begin{itemize}
        \item The skew brace on $\Z/4\Z$, with multiplication given by $x\circ y=x+y+2xy$.
        \item The skew brace on $\Z/2\Z\times\Z/2\Z$, with multiplication 
        \[
        (x_1,y_1)\circ(x_2,y_2)=(x_1+x_2+y_1y_2, y_1+y_2).
        \]
        \item The skew brace on $\Z/2\Z\times\Z/4\Z$, with multiplication 
        \[
        (x_1,y_1)\circ (x_2,y_2)=\left(x_1+x_2+y_2\sum\limits_{i=1}^{y_1-1}i,y_1+y_2+2y_1y_2
        \right).
        \]
    \end{itemize}
\end{thm}
\begin{proof}
 By Proposition~\ref{1vertex:2brace:pro} and Corollary~\ref{1vertex:ordermax8:pro}, $|B|\in\{4,8\}$. 
 Using Proposition~\ref{p2graph:pro}, we can say immediately that both non-trivial skew braces of order 4 have a one-vertex graph. 
 Thus, using Theorem~\ref{thm:p2} 
 we get the first two cases.
 
Assume now $|B|=8$. Let $\overline{B}=B/\Soc(B)$.
 Proposition~\ref{pro:1vertex quotient} implies that $(\Fix(B),+)\cong(\overline{B},\circ)\cong \Z/2\Z\times\Z/2\Z$.
 By Lemma~\ref{1vertexabelian:oreder4:lem}, $(B,\circ)$ has to be non-abelian, so it is isomorphic either to $D_8$ or $Q_8$.
 Moreover, $\Soc(B)$ is a central subgroup of $(B,\circ)$, and in both of the previous cases, $\left|Z(B,\circ)\right|=2$ and $(B,\circ)/Z(B,\circ)\cong\Z/2\Z\times \Z/2\Z$.
 Thus $(\overline{B},\circ)$ is isomorphic to  $\Z/2\Z\times \Z/2\Z$ and, by Theorem~\ref{thm:p2}, $(\overline{B},+)$ is isomorphic to $\Z/4\Z$.
Since $(B,+)$ has to have a subgroup isomorphic to $\Z/2\Z\times \Z/2\Z$ and a quotient isomorphic to $\Z/4\Z$, the only possibility is that $(B,+)\cong \Z/2\Z\times\Z/4\Z$.
Finally, looking at the classification of skew braces of abelian type of order 8 in~\cite{MR3320237},
there are only two possible skew brace structures on $\Z/2\Z\times\Z/4\Z$ with socle of order two and non-abelian multiplicative group:
the skew brace $B_1$ on $\Z/2\Z\times\Z/4\Z$ with
 \[
(x_1,y_1)\circ(x_2,y_2)=(x_1+x_2,y_1+y_2+2(x_1+y_1)x_2+2y_1y_2)
 \]
 and the skew brace $B_2$ on $\Z/2\Z\times\Z/4\Z$ with
\[
(x_1,y_1)\circ(x_2,y_2)=\left(x_1+x_2+y_2\sum\limits_{i=1}^{y_1-1}i,y_1+y_2+2y_1y_2\right).
\]
Note that 
\[
\sum\limits_{i=1}^{y_1-1}i=\begin{cases}0&\hbox{ if }y_1\in\{0,1\},\\1&\hbox{ if }y_1\in\{2,3\}.\end{cases}
\]
In the first case, 
\[
(x_1,y_1)*(x_2,y_2)=(0,2(x_1+y_1)x_2+2y_1y_2).
 \]
Thus $(x_2,y_2)\in \Fix(B_1)$ if and only if $2(x_1+y_1)x_2+2y_1y_2=0$ in $\Z/4\Z$ for all $x_1\in \Z/2\Z$ and $y_1\in\Z/4\Z$.
 This means that $(x_2,y_2)\in \Fix(B_1)$ if and only if $2y_2=0$ in $\Z/4\Z$ and $2x_2=0$ in $\Z/4\Z$.
 Hence $\Fix(B_1)=\{0\}\times 2\Z/4\Z$ has order two and $\Lambda(B_1)$ cannot have only one vertex.
In the second case, 
\[(x_1,y_1)* (x_2,y_2)=\left(y_2\sum\limits_{i=1}^{y_1-1}i,2y_1y_2\right).
\]
Thus $(x_2,y_2)\in \Fix(B_2)$ if and only if $y_2\sum\limits_{i=1}^{y_1-1}i=0$ in $\Z/2\Z$ and $2y_1y_2=0$ in $\Z/4\Z$ for all $y_1\in\Z/4\Z$.
 This means that $(x_2,y_2)\in \Fix(B_2)$ if and only if $y_2=0$ in $\Z/2\Z$ and $2y_2=0$ in $\Z/4\Z$.
 Hence $\Fix(B_2)=\Z/2\Z\times 2\Z/4\Z$.
 Moreover, 
 \[
 \lambda_{(x_1,y_1)}(0,1)=(0,3y_2),
 \]
 thus $\Lambda(0,1)=\{0\}\times \Z/4\Z=B_2\setminus\Fix(B_2)$ and $\Lambda(B_2)$ has only one vertex.
\end{proof}

\section{Skew braces with one-vertex \texorpdfstring{$\lambda$}{lambda}-graph: Finale}
In this section, we obtain a full classification of skew braces with a $\lambda$-graph with one vertex,
showing that they consist of four families of skew braces.

\begin{thm}
\label{thm: onevertex:general:structure}
Let $A$ be a finite skew brace.
$\Lambda(A)$ has exactly one vertex if and only if 
$A$ is isomorphic to a skew brace on $F\times \Z/2\Z$ with operations
\begin{equation}
    \label{eq: onevertex operations}
\begin{aligned}
    (f_1,k_1)+(f_2,k_2)&=\big(f_1+(-1)^{k_1}f_2+k_1k_2y,k_1+k_2\big)\\
    (f_1,k_1)\circ(f_2,k_2)&=\big(f_1+(-1)^{k_1}f_2+\psi(f_1,k_1,k_2)+k_1k_2y,k_1+k_2\big)
\end{aligned}
\end{equation}
where $F\neq\{0\}$ is an abelian group,
$y\in F$ is an element such that $2y=0$, and
$\psi\colon F\times\Z/2\Z\times\Z/2\Z\to F$ is a surjective map such that
\[
\psi(f_1,k_1,k_2)=\frac{1-(-1)^{k_2}}{2}\left(\phi(f_1)-\frac{1-(-1)^{k_1}}{2}z\right),
\]
where $\phi\in\End(F)$, $z\in F$,
$\phi(z)=\phi(y)-2z$,
and $\phi(\phi(f))=-2\phi(f)$ for all $f\in F$.
\end{thm}
\begin{proof}
    Let $A$ be a skew brace with $\Lambda(A)$ with exactly one vertex
    and denote $F=\Fix(A)$.
    By Lemma ~\ref{1vertex:A/ker=Fix:lem} we get directly that 
    $(F,+)$ is abelian and
    $(A,+)$ is isomorphic to $(F\times \Z/2\Z,+)$ with 
    \[
    (f_1,k_1)+(f_2,k_2)=(f_1+(-1)^{k_1}f_2+c(k_1,k_2),k_1+k_2),
    \]
    where $c\colon\Z/2\Z\times\Z/2\Z\to F$ is a normalized 2-cocycle, i.e. $y=c(1,1)\in F$ is such that $2y=0$ and 
    $c(i,j)=0$ for every $(i,j)\neq (1,1)$.
    Moreover in this isomorphism $\Fix(A)$ corresponds to $F\times \{0\}$ 
    and the only non-trivial $\lambda$-orbit corresponds to $F\times\{1\}$.
    Therefore, translating the $\lambda$-action using this isomorphism,
    $\lambda_{(f_1,k_1)}$ is the identity on $F\times\{0\}$ and $\lambda_{(f_1,k_1)}(F\times\{1\})=F\times\{1\}$ for every $(f_1,k_1)\in F\times\Z/2\Z$.
    Thus there exists a map $h:F\times\Z/2\Z\to F$ such that $\lambda_{(f_1,k_1)}(0,1)=(h(f_1,k_1),1)$ for every $(f_1,k_1)\in F\times \Z/2\Z$. Hence 
    \[
    \lambda_{(f_1,k_1)}(f_2,k_2)=\begin{cases}
        (f_2,0)&\text{ if } k_2=0\\
        (f_2,0)+(h(f_1,k_1),1)=(f_2+h(f_1,k_1),1)&\text{ if } k_2=1
    \end{cases}
    \]
    for every $(f_1,k_1), (f_2,k_2)\in F\times \Z/2\Z$.
    Define $\phi:F\to F$ as $\phi(f)=h(f,0)$ and $z=h(0,1)$.
    Since for every $f_1,f_2\in F$
    \begin{align*}
        (\phi(f_1)+\phi(f_2),1)&=
    \lambda_{(f_1,0)}(\lambda_{(f_2,0)}(0,1))\\
    &=
    \lambda_{(f_1,0)+\lambda_{(f_1,0)}(f_2,0)}(0,1)\\
    &=(\phi(f_1+f_2),1),
    \end{align*}
    $\phi$ is a group homomorphism.
    Moreover for every $f\in F$
    \[
    (\phi(f)+h(0,1),1)=
    \lambda_{(0,1)}(\lambda_{(f,0)}(0,1))=
    \lambda_{(0,1)+\lambda_{(0,1)}(f,0)}(0,1)=
    (h(-f,1),1),
    \]
    thus $h(f,1)=z-\phi(f)$ for all $f\in F$.
    We also have that for every $f,f_1\in F$
    \begin{align*}
    (h(f_1,1)+\phi(f),1)&=
    \lambda_{(f,0)}(\lambda_{(f_1,1)}(0,1))\\
    &=\lambda_{(f,0)+\lambda_{(f,0)}(f_1,1)}(0,1)\\ 
    &=(h(f+f_1+\phi(f),1),1),    
    \end{align*}
    hence 
    \[
    z-\phi(f_1)+\phi(f)=
    h(f_1,1)+\phi(f)=
    h(f+f_1+\phi(f),1)=
    z-\phi(f)-\phi(f_1)-\phi(\phi(f)).
    \]
    Therefore $\phi(\phi(f))=-2\phi(f)$ for all $f\in F$. Finally    
    \begin{align*}
        (2z,1)=\lambda_{(0,1)}(z,1)&=
    \lambda_{(0,1)}(\lambda_{(0,1)}(0,1))\\
    &=\lambda_{(0,1)+\lambda_{(0,1)}(0,1)}(0,1)\\
    &=(-\phi(z)+y,1)
    \end{align*}  
    hence $\phi(z)=\phi(y)-2z$.
    Thus we have the thesis defining $\psi(f,k,0)=0$ and $\psi(f,k,1)=(-1)^kh(f,k)$ for all $(f,k)\in F\times \Z/2\Z$.

    Conversely, consider on $A=F\times \Z/2\Z$ the operations defined as in the statement.
    Since $(k_1,k_2)\mapsto k_1k_2 y$ is a 2-cocycle 
    (for the action $k\cdot f=-f$ of $\Z/2\Z$ on $F$),
    we know that $(A,+)$ is a group.
    As for the $\circ$ operation,
    it is easy to see that $(0,0)$ is the neutral element.
    We claim that $(A,\circ)$ is a group, so we need to prove the associativity.
    For all $(f,k),(f_1,k_1),(f_2,k_2)\in A$, one can see that
\[
(f,k)\circ\big((f_1,k_1)\circ(f_2,k_2)\big)=\big((f,k)\circ(f_1,k_1)\big)\circ(f_2,k_2)
\] is equivalent to 
\begin{equation}
    \label{eq:associativity}
    \begin{aligned}
    &(-1)^k\psi(f_1,k_1,k_2)+\psi(f,k,k_1+k_2)\\
        &=\psi(f,k,k_1)+\psi 
        \big(f+(-1)^kf_1+\psi(f,k,k_1)+kk_1y,k+k_1,k_2\big).
    \end{aligned}
    \end{equation}
If $k_2=0$, then \eqref{eq:associativity} follows easily. If $k_2=1$ and $k_1=0$, then \eqref{eq:associativity} becomes
\[
(-1)^k\phi(f_1)+\psi(f,k,1)=\psi 
        (f+(-1)^kf_1,k,1)
\]
and this holds for all $f,f_1\in F$ and $k\in \Z/2\Z$
thanks to the definition of $\psi$ and to the fact the $\phi$ is a group homomorphism. Finally, if $k_2=k_1=1$, using that $\psi(a,h,1)+\psi(b,h+1,1)=\phi(a)+\phi(b)-z$ for all $a,b\in F$ and $h\in\Z/2\Z$, \eqref{eq:associativity} becomes
\[
(-1)^{k+1}z=2\phi(f)+\phi(\psi(f,k,1))+\phi(c(k,1))-z.
\]
Which is equivalent to the hypotheses $\phi(\phi(f))=-2\phi(f)$ and $\phi(z)=\phi(y)-2z$.

To prove that $(A,+,\circ)$ is a skew brace we have to show 
that
\[
(f,k)\circ\big((f_1,k_1)+(f_2,k_2)\big)=(f,k)\circ(f_1,k_1)-(f,k)+(f,k)\circ (f_2,k_2)
\]
for all $(f,k), (f_1,k_1), (f_2,k_2)\in A$.
This is equivalent to proving that
\begin{align*}
    f&+(-1)^k\big(f_1+(-1)^{k_1}f_2+k_1k_2y\big)+\psi(f,k,k_1+k_2)+k(k_1+k_2)y\\
    =&f+(-1)^kf_1+\psi(f,k,k_1)+kk_1y+(-1)^{k_1+1}f+(-1)^{k+k_1}k^2y+(k+k_1)ky\\
    &+(-1)^{k_1}\big(f+(-1)^kf_2+\psi(f,k,k_2)+kk_2y\big)+k_1(k+k_2)y,
\end{align*}
for all $f,f_1,f_2\in F$ and $k,k_1,k_2\in \Z/2\Z$,
which means
\[
\psi(f,k,k_1+k_2)=\psi(f,k,k_1)+(-1)^{k_1}\psi(f,k,k_2),
\]
that is a trivial equality.

Finally we need to compute $\Lambda(A)$
and so the $\lambda$-action of $A$.
Note that for all $(f_1,k_1), (f_2,k_2)\in A$,
\begin{align*}
\lambda&_{(f_1,k_1)}(f_2,k_2)\\
&=\big((-1)^{k_1+1}f_1+k_1y,k_1\big)+\big(f_1+(-1)^{k_1}f_2+\psi(f_1,k_1,k_2)+k_1k_2y,k_1+k_2\big)\\
&=\big(f_2+(-1)^{k_1}\psi(f_1,k_1,k_2),k_2\big).
\end{align*}
Therefore $F\times\{0\}\subseteq \Fix(A)$ and $(f_2,1)\in \Fix(A)$ if and only if $0=\psi(f_1,k_1,1)$ for all $(f_1,k_1)\in A$, a contradiction with the surjectivity of $\psi$. Hence $\Fix(A)=F\times\{0\}$. Moreover, again by surjectivity of $\psi$,
\[
\Lambda(0,1)=\left\{\big((-1)^{k_1}\psi(f_1,k_1,1),1\big)\colon (f_1,k_1)\in A\right\}=F\times\{1\}.
\]
Thus $\Lambda(A)$ has exactly one vertex.
\end{proof}

\begin{thm}
\label{thm: onevertex isomorphic condition}
    Let $F\ne\{0\}$ be a finite abelian group
    and let $A$ and $A_1$ be two skew braces on $F\times\Z/2\Z$
    defined as in Theorem ~\ref{thm: onevertex:general:structure} with $\psi, \phi, y, z$ used for the definition of $A$ and 
    $\psi_1, \phi_1, y_1, z_1$ used for the definition of $A_1$.
    Then $A$ and $A_1$ are isomorphic
    if and only if there exist $\sigma\in\Aut(F)$
    such that $\phi_1=\sigma\circ\phi\circ\sigma^{-1}$,
    $\sigma(y)=y_1$ and $z_1-\sigma(z)\in\phi_1(F)$.    
\end{thm}
\begin{proof}
We first prove $\implies$. 
Suppose there is a map $\Theta:F\times\Z/2\Z\to F\times\Z/2\Z$ which is 
an isomorphism of skew braces from $A$ to $A_1$.
We will use the notation $\Theta(f,k)=(\alpha(f,k),\beta(f,k))$.
Since $\Theta(F\times\{0\})=\Theta(\Fix(A))=\Fix(B)=F\times \{0\}$,
then $\beta(f,0)=0$ for all $f\in F$.
Moreover, since $(f,0)+(0,1)=(f,1)$, we have that  for all $f\in F$
\begin{align*}
    \big(\alpha&(f,1),\beta(f,1)\big)=
    (\alpha(f,0),\beta(f,0))+_1(\alpha(0,1),\beta(0,1))\\
    &=(\alpha(f,0),0)+(\alpha(0,1),\beta(0,1))=\big(\alpha(f,0)+\alpha(0,1),\beta(0,1)\big),
\end{align*}
i.e. $\beta(f,1)=\beta(0,1)$ for all $f\in F$.
By surjectivity of $\Theta$ and the fact that $\beta(f,0)=0$, we get that $\beta(f,1)=\beta(0,1)=1$ and so $\beta(f,k)=k$ for all $f\in F$.
Moreover, since $\Theta$ is a group homomorphism for the additive groups,
$\alpha(y,0)=y_1$ and
\[
\alpha(f_1+f_2,0)=\alpha(f_1,0)+\alpha(f_2,0),\quad
\alpha(f_1-f_2,1)=\alpha(f_1,1)-\alpha(f_2,0)
\]
 for all $f_1,f_2\in F$.
 Hence, defining $\sigma:F\to F$ as the map $f\mapsto\alpha(f,0)$
and $s=\alpha(0,1)$,
we have that $\sigma\in\End(F)$,
$\sigma(y)=y_1$,
and $\alpha(f,1)=\sigma(f)+s$ for all $f\in F$.
Moreover, for all $(f,k)\in F\times\Z/2\Z$,
\[
\alpha(f,k)=\sigma(f)+\frac{1-(-1)^k}{2}s.
\]
Thus $(0,0)=\Theta(f,k)=\left(\sigma(f)+\frac{1-(-1)^k}{2}s,k\right)$ if and only if 
$k=0$ and $\sigma(f)=0$, i.e. $\{(0,0)\}=\ker\Theta=\ker\sigma\times\{0\}$.
Hence $\sigma\in \Aut(F)$ and $\sigma(y)=y_1$.
Furthermore, since $\Theta$ is a skew brace isomorphism, it is compatible with the $\lambda$-actions. Hence, recalling that
$\lambda_{(f_1,k_1)}(f_2,k_2)=(f_2+(-1)^{k_1}\psi(f_1,k_1,k_2),k_2)$ and analogously for 
the $\lambda$-action $\lambda^1$ of $A_1$,
\[
\alpha(f_2,k_2)+(-1)^{k_1}\psi_1(\alpha(f_1,k_1),k_1,k_2)=\alpha(f_2+(-1)^{k_1}\psi(f_1,k_1,k_2),k_2)
\]
for all $(f_1,k_1),(f_2,k_2)\in F\times\Z/2\Z$.
In particular,
\[
s+\phi_1(\sigma(f))=\alpha(0,1)+\psi_1(\alpha(f,0),0,1)=\alpha(\psi(f,0,1),1)=\sigma(\phi(f))+s,
\]
 for all $f\in F$, and
\[
s-\phi_1(s)+z_1=\alpha(0,1)-\psi_1(\alpha(0,1),1,1)=\alpha(-\psi(0,1,1),1)=\sigma(z)+s
\]
i.e. $\phi_1=\sigma\circ\phi\circ\sigma^{-1}$ and $z_1-\sigma(z)=\phi_1(s)\in \phi_1(F)$.

We now prove $\impliedby$.
Let $s\in F$ be such that $z_1-\sigma(z)=\phi_1(s)$ and
define $\Theta:F\times\Z/2\Z\to F\times\Z/2\Z$ as 
\begin{align*}
   \Theta(f,k)&=\begin{cases}
(\sigma(f),0)&\text{ if }k=0\\
(\sigma(f)+s,1)&\text{ if }k=1
\end{cases}\\
&=\left(\sigma(f)+\frac{1-(-1)^k}{2}s,k\right). 
\end{align*}

for all $(f,k)\in F\times\Z/2\Z$.
We claim that $\Theta$ is an isomorphism of skew braces from $A$ to $A_1$.
\begin{align*}
&\Theta(f_1,k_1)+_1\Theta(f_2,k_2)\\ 
&=\left(\sigma(f_1)+\frac{1-(-1)^{k_1}}{2}s,k_1\right)+_1\left(\sigma(f_2)+\frac{1-(-1)^{k_2}}{2}s,k_2\right)\\
&=\left(\sigma(f_1)+\frac{1-(-1)^{k_1}}{2}s+(-1)^{k_1}\left(\sigma(f_2)+\frac{1-(-1)^{k_2}}{2}s\right)+k_1k_2y_1,k_1+k_2\right)\\
&=\left(\sigma(f_1+(-1)^{k_1}f_2+k_1k_2y)+\frac{1-(-1)^{k_1+k_2}}{2}s,k_1+k_2\right)\\
&=\Theta\big(f_1+(-1)^{k_1}f_2+k_1k_2y,k_1+k_2\big)\\
&=\Theta\big((f_1,k_1)+(f_2,k_2)\big),
\end{align*}
for all $(f_1,k_1),(f_2,k_2)\in F\times\Z/2\Z$.
So $\Theta$ is a homomorphism between the additive groups.
Moreover, observe that for all $f\in F$ and $k_1,k_2\in\Z/2\Z$
\begin{align*}
    \sigma\left(\psi(f,k_1,k_2)\right)&
    =\frac{1-(-1)^{k_2}}{2}\left(\phi_1(\sigma(f_1))-\frac{1-(-1)^{k_1}}{2}(z_1-\phi_1(s))\right)\\
    &=\psi_1\left(\sigma(f)+\frac{1-(-1)^{k_1}}{2}s,k_1,k_2\right).
\end{align*}
Thus for all $(f_1,k_1),(f_2,k_2)\in F\times \Z/2\Z$
\begin{align*}
&\Theta(f_1,k_1)\circ_1\Theta(f_2,k_2)\\
&=\left(\sigma(f_1)+\frac{1-(-1)^{k_1}}{2}s,k_1\right)\circ_1\left(\sigma(f_2)+\frac{1-(-1)^{k_2}}{2}s,k_2\right)\\
&=\Bigg(\sigma(f_1)+(-1)^{k_1}\sigma(f_2)+\psi_1\left(\sigma(f_1)+\frac{1-(-1)^{k_1}}{2}s,k_1,k_2\right)+\\
&\qquad+k_1k_2 y_1+\frac{1-(-1)^{k_1}+(-1)^{k_1}(1-(-1)^{k_2})}{2}s,k_1+k_2\Bigg)\\
&=\left(\sigma\left(f_1+(-1)^{k_1}f_2+\psi(f_1,k_1,k_2)+k_1k_2y\right)+\frac{1-(-1)^{k_1+k_2}}{2}s,k_1+k_2\right)\\
&=\Theta\big(f_1+(-1)^{k_1}f_2+\psi(f_1,k_1,k_2)+k_1k_2y,k_1+k_2\big)\\
&=\Theta\big((f_1,k_1)\circ(f_2,k_2)\big).
 \end{align*}
Therefore $\Theta$ is a skew brace homomorphism.
Moreover $(f,k)\in\ker\Theta$ if and only if
$\left(\sigma(f)+\frac{1-(-1)^{k}}{2}s,k\right)=(0,0)$, i.e. $k=0$ and $\sigma(f)=0$.
Thus, since $\sigma$ is an automorphism, $\ker\Theta=\ker\sigma\times\{0\}=\{(0,0)\}$ and $\Theta$ is an isomorphism.
\end{proof}

\begin{lem}
\label{lem:onesize map}
Let $F\neq\{0\}$ be a finite abelian group,
and $\phi\in\End(F)$, 
such that $\left|\ker\phi\right|\leq2$ and
$\phi^2=-2\phi$.

Then $F\cong \Z/2^i\Z\times\Z/2^j\Z\times G$, 
for some $j\in\{0,1\}$ and $i\geq j$
and where $G$ is an abelian group of odd size.

Moreover $\phi=(\alpha,-2\id_{G})$, where $\alpha\in \End(P_2)$ such that $\alpha(P_2)$ is cyclic of index 2 and $\alpha^2=-2\alpha$.
\end{lem}
\begin{proof}
    $F$ is a finite abelian group, thus $F\cong P_2\times G$,
    where $P_2$ is the Sylow 2-subgroup of $F$.
    We can also see $\phi=(\alpha,\beta)\in \End(P_2)\times \End(G)$.
    By assumption $2\geq \left|\ker\phi\right|=\left|\ker\alpha\right|\left|\ker\beta\right|$, but $2\nmid |G|$, thus $\beta\in\Aut(G)$ and $\left|\ker\phi\right|\leq 2$.
    Since $\beta^2(g)=-2\beta(g)$ for all $g\in G$,
    then $\beta(g)=-2g$ for all $g\in G$.
    Moreover, we can write $P_2=\bigtimes_{i\geq 1}\Z_{2^i}^{n_i}$ for some $n_i\geq0$ and denote 
    by $T=\bigtimes_{i\geq 1}(2^{i-1}\Z_{2^i})^{n_i}$ the subgroup of involutions.
    Since $\alpha^2(x)=-2\alpha(x)$ for all $x\in P_2$,
    we have that $\alpha\notin \Aut(P_2)$ and
    $\alpha(P_2)\cap T\subseteq \ker\alpha.$
     Hence, $[P_2:\alpha(P_2)]=\left|\ker\alpha\right|=2$ and
     \[
     \frac{\prod_{i\geq 1}2^{n_i}}{2}=\frac{|T|}{2}=\frac{\left|\alpha(P_2)\right||T|}{|P_2|}\leq \frac{\left|\alpha(P_2)\right||T|}{\left|\alpha(P_2)+T\right|}=\left|\alpha(P_2)\cap T\right|\leq \left|\ker\alpha\right|=2.
     \]
    Therefore there exist $i\geq j\geq 0$ such that $P_2=\Z/2^i\Z\times\Z/2^j\Z$.
    Moreover, the number of involutions of $\alpha(P_2)$ is
    $\left|\alpha(P_2)\cap T\right|\leq \left|\ker\alpha\right|=2$.
    Thus $\alpha(P_2)$ is cyclic of index $2$
    and so $j\leq 1$.
\end{proof}

\begin{lem}
\label{lem:onevertex:restriction}
Let $F\neq\{0\}$ be an abelian group,
$y\in F$ such that $2y=0$, and
let $\psi\colon F\times\Z/2\Z\times\Z/2\Z\to F$ be the map
\[
\psi(f,k_1,k_2)=\frac{1-(-1)^{k_2}}{2}\left(\phi(f)-\frac{1-(-1)^{k_1}}{2}z\right),
\]
where $\phi\in\End(F)$, $z\in F$.
Suppose that $\psi$ is surjective,
$\phi(z)=\phi(y)-2z$,
and $\phi^2=-2\phi$.
Then there exists an abelian group $G$ of odd order such that
either $F\cong \Z/2\Z/\times \Z/2\Z\times G$
and $\phi=(\alpha,-2\id_G)$, for $\alpha\in\End(\Z/2\Z/\times \Z/2\Z)$
such that $\alpha(P_2)=\ker\alpha$ and it is a subgroup of order 2,
or $F\cong \Z/2^i\Z\times G$ and $\phi=-2\id_F$.
\end{lem}
\begin{proof}
Since $\psi(F)=\phi(F)\cup (\phi(F)-z)$ and
$\psi$ is surjective, $\left|\ker\phi\right|\leq 2$. 
Moreover $\phi^2(f)=-2\phi(f)$ for all $f\in F$.
Therefore, by Lemma~\ref{lem:onesize map}, $F=P_2\times G$,
where $P_2=\Z/2^i\Z\times\Z/2^j\Z$ for some $j\in\{0,1\}$ and $i\geq j$,
$G$ is an abelian group of odd size and $\phi=(\alpha,-2\id_G)$. 
If $i=0$, then $F=G$ and $\phi=-2\id_G$.

Let $i>0$. The subgroup of involutions of $F$ is
$T=2^{i-1}\Z/2^i\Z\times\Z/2^j\Z\times \{0\}$.
If $x\in \phi(F)=\alpha(P_2)\times G$, then $\phi(x)=-2x$.
Otherwise, by the surjectivity of $\psi$, it follows that $\phi(F)\cup (\phi(F)-z)=F$. Hence, we have that
$x+z\in \phi(F)$ and 
\[
\phi(x)=-2(x+z)-\phi(z)=-2(x+z)-\phi(y)+2z=-2x-\phi(y).
\]
Since $y\in T$, we have that $\alpha(P_2)\subseteq 2P_2+T\subseteq
\begin{cases}
    \Z/2\Z\times\Z/2^j\Z\times G&\text{ if }i=1\\
   2\Z/2^i\Z\times\Z/2^j\Z\times G &\text{ if }i>1
\end{cases}.$
But, by Lemma~\ref{lem:onesize map}, $\alpha(P_2)$ is a cyclic group of index 2,
thus either $j=0$ or $i=j=1$. If $i=j=1$, then $P_2=\Z/2\Z\times \Z/2\Z$ and $\alpha^2=0$, so $\alpha(P_2)=\ker\alpha$.

If $j=0$, then $P_2=\Z/2^i\Z$ and $\alpha(x)=lx$ for some $l\in \Z/2^i\Z$.
We also know that $\alpha(P_2)=2\Z/2^i\Z$ and $\ker\alpha=2^{i-1}\Z/2^i\Z$,
hence $l=2d$ for some $d\in \Z/2^i\Z$ such that $d\equiv1\pmod{2}$.
We can write $y=(2^{i-1}y_1,0)\in P_2\times G$ and $z=(z_1,z_2)\in (P_2\times G)$.
Since $z_1\in P_2\setminus\alpha(P_2)$, we have that $z_1\equiv 1\pmod{2}$.
Finally, $\phi(z)=\phi(y)-2z$ implies
$(2dz_1,-2z_2)=(2^idy_1-2z_1,-2z_2)$, so either $i=1$ and $l=0$ or $i>1$ and $d=-1$.
In any case, $l=-2\pmod{2^i}$ and $\phi=-2\id_F$.
\end{proof}

\begin{pro}
\label{pro:onevertex 8d}
    Let $A$ be a finite skew brace such that $\Lambda(A)$ has only one vertex.
    If $\Fix(A)\cong\Z/2\Z\times\Z/2\Z\times G$,
    for some abelian group $G$ of odd order $d$,
    then $A$ is isomorphic to the skew brace $K_{8d}(G)$
    on $\Z/2\Z\times\Z/2\Z\times G\times \Z/2\Z$ with
\begin{align*}
    &(a_1,b_1,g_1,k_1)+(a_2,b_2,g_2,k_2)\\
    &=\big(a_1+a_2,b_1+b_2+k_1k_2,g_1+(-1)^{k_1}g_2,k_1+k_2\big)
\end{align*}
and
\begin{align*}
    &(a_1,b_1,g_1,k_1)\circ(a_2,b_2,g_2,k_2)\\
    &=\big(a_1+a_2+k_2b_1,b_1+b_2,(-1)^{k_2}g_1+(-1)^{k_1}g_2,k_1+k_2\big),
 \end{align*}
 for all $(a_1,b_1,g_1,k_1),(a_2,b_2,g_2,k_2)\in \Z/2\Z\times\Z/2\Z\times G\times \Z/2\Z$.

 Moreover, for any odd number $d$ and any abelian group $G$ of order $d$ the graph $\Lambda(K_{8d}(G))$ has only one vertex and
 $\Fix(K_{8d}(G))\cong\Z/2\Z\times\Z/2\Z\times G$.
\end{pro}
\begin{proof}
   By Theorem ~\ref{thm: onevertex:general:structure}, $A$ is isomorphic to
   a skew brace on the set $F\times\Z/2\Z$, for $F=\Z/2\Z\times\Z/2\Z\times G$,
   with operations defined by Equation ~\eqref{eq: onevertex operations} in Theorem ~\ref{thm: onevertex:general:structure}.
   
   The element $y\in \Z/2\Z\times\Z/2\Z\times G$ required for the definition of the addition
   has to be an involution, i.e. $y=(y_1,y_2,0)$, for some $y_1,y_2\in\Z/2\Z$.
   By  Lemma ~\ref{lem:onevertex:restriction}, the endomorphism $\phi$
   that defines the multiplication is $\phi=(\alpha,-2\id_G)$, for some $\alpha\in \End(\Z/2\Z\times\Z/2\Z)$ such that $\alpha(\Z/2\Z\times\Z/2\Z)=\ker\alpha$ and $\left|\ker\alpha\right|=2$.  
   Thus for all $x=(x_1,x_2,x_3)\in \Z/2\Z\times\Z/2\Z\times G$ and $k_1,k_2\in \Z/2\Z$
   \begin{align*}
       &\psi((x_1,x_2,x_3),k_1,k_2)\\
       &=\frac{1-(-1)^{k_2}}{2}\left((\alpha(x_1,x_2),-2x_3)-\frac{1-(-1)^{k_1}}{2}(z_1,z_2,z_3)\right),
   \end{align*}
    where $z=(z_1,z_2,z_3)\in \Z/2\Z\times\Z/2\Z\times G$
    such that $\psi$ is surjective,
    which means $z\notin \phi(\Z/2\Z\times\Z/2\Z\times G)=\alpha(\Z/2\Z\times\Z/2\Z)\times G$.
   The condition $\phi(z)=\phi(y)-2z$ becomes $\alpha(z_1,z_2)=\alpha(y_1,y_2)$.

   Let $(h_1,h_2)\in\Z/2\Z\times\Z/2\Z\setminus\{(0,0)\}$ such that 
   \[
   H=\langle (h_1,h_2)\rangle=\alpha(\Z/2\Z\times\Z/2\Z)=\ker\alpha.
   \]
   Then $(z_1,z_2)\notin H$ and so $\alpha(z_1,z_2)\ne(0,0)$. 
   
   Hence $\alpha(y_1,y_2)=\alpha(z_1,z_2)=(h_1,h_2)$ and $(z_1,z_2)-(y_z,y_2)\in H$.
   Observe that we obtain the skew brace $K_{8d}(G)$ choosing $\alpha:(x_1,x_2)\mapsto (x_2,0)$ and $y=(0,1,0)=z$.
   
   Fix any skew brace $A_1$ defined using $\alpha$ with
   $H=\langle(h_1,h_2)\rangle$, $y=(y_1,y_2,0)$ and $z=(z_1,z_2,z_3)\notin H\times G$.
   We claim that the map
   \[
   \sigma:(x_1,x_2,x_3)\mapsto (x_1(h_1,h_2)+x_2(y_1,y_2),x_3)
   \]
   satisfies the conditions
   of Theorem ~\ref{thm: onevertex isomorphic condition} for $A=K_{8d}(G)$ and $A_1$.
   First of all $\sigma(0,1,0)=(y_1,y_2,0)$.
   Moreover, since $(z_1,z_2)-(y_z,y_2)\in H$,
   \[
   (z_1,z_2,z_3)-\sigma(0,1,0)=(z_1,z_2,z_3)-(y_1,y_2,0)\in H\times G.
   \]
   And finally, recalling that $\alpha(y_1,y_2)=\alpha(z_1,z_2)=(h_1,h_2)$, 
   we have that for all $(x_1,x_2,x_3)\in \Z/2\Z\times\Z/2\Z\times G$
   \begin{align*}
    \phi\sigma(x_1,x_2,x_3)&=\big(\alpha(x_1(h_1,h_2)+x_2(y_1,y_2)),-2x_3\big)=(x_2(h_1,h_2),-2x_3)\\
    &=\sigma(x_2,0,-2x_3).   
   \end{align*}
      Thus $A_1$ is isomorphic to $K_{8d}(G)$.
    Conversely, let $d$ be an odd number and $G$ an abelian group of order $d$.
    By the construction above and using Theorem~\ref{thm: onevertex:general:structure}, $\Lambda(K_{8d}(G))$ has only one vertex and $\Fix(K_{8d}(G))=\Z/2\Z\times\Z/2\Z\times G\times\{0\}$.   
\end{proof}

\begin{rem}
Let $G$ and $G'$ be abelian groups of odd order $d$, then $G$ is isomorphic to $G'$ if and only if $K_{8d}(G)$ and $K_{8d}(G')$ are isomorphic skew braces. 
\end{rem}

\begin{pro}
\label{pro:onevertex 2^id}
    Let $A$ be a finite skew brace such that $\Lambda(A)$ has only one vertex.
    If $\Fix(A)\cong\Z/2^i\Z\times G$, for $i>0$ and
    for some abelian group $G$ of odd order $d$,
    then as a skew brace $A$ is isomorphic to $J_{2^{i+1}d}(G)$ on the set $\Z/2^i\Z\times G\times \Z/2\Z$ with
\[
    (a_1,g_1,k_1)+(a_2,g_2,k_2)=\big(a_1+(-1)^{k_1}a_2,g_1+(-1)^{k_1}g_2,k_1+k_2\big)    
\]
and
\begin{align*}
    &(a_1,g_1,k_1)\circ(a_2,g_2,k_2)\\
    &=\left((-1)^{k_2}a_1+(-1)^{k_1}a_2-\frac{1-(-1)^{k_1k_2}}{2},(-1)^{k_2}g_1+(-1)^{k_1}g_2,k_1+k_2\right)
\end{align*}
or
$A$ is isomorphic as a skew brace to $H_{2^{i+1}d}(G)$ on the set $\Z/2^i\Z\times G\times \Z/2\Z$ with
\[
    (a_1,g_1,k_1)+(a_2,g_2,k_2)=\big(a_1+(-1)^{k_1}a_2+2^{i-1}k_1k_2,g_1+(-1)^{k_1}g_2,k_1+k_2\big)
\]
and
\begin{align*}
    &(a_1,g_1,k_1)\circ(a_2,g_2,k_2)\\
    &=\left(a_1+(-1)^{k_1}a_2-\frac{1-(-1)^{k_1k_2}}{2}+2^{i-1}k_1k_2,(-1)^{k_2}g_1+(-1)^{k_1}g_2,k_1+k_2\right).
\end{align*}
Moreover $J_{2^{i+1}d}(G)$ and $H_{2^{i+1}d}(G)$ are not isomorphic.
\end{pro}
\begin{proof}
   By Theorem ~\ref{thm: onevertex:general:structure}, $A$ is isomorphic to
   a skew brace on $\Z/2^i\Z\times G\times\Z/2\Z$ 
   with operations defined by Equation ~\eqref{eq: onevertex operations} (Theorem ~\ref{thm: onevertex:general:structure}).
   
   The element $y\in \Z/2^i\Z\times G$ required for the definition of the addition
   has to be an involution, i.e. $y\in \{(0,0),(2^{i-1},0)\}$. By  Lemma ~\ref{lem:onevertex:restriction}, we have that the map $\phi\in\End(\Z/2^i\Z)$
   that defines the multiplication is $\phi=-2\id$. Thus for all $f\in \Z/2^i\Z$ and $k_1,k_2\in \Z/2\Z$
   \begin{align*}
       \psi(f,k_1,k_2)&=\frac{1-(-1)^{k_2}}{2}\left(-2f-\frac{1-(-1)^{k_1}}{2}z\right)\\
       &=(-1+(-1)^{k_2})f-\frac{1-(-1)^{k_1k_2}}{2}z,
   \end{align*}
   
   where $z\in \Z/2^i\Z\times G$ such that $\psi$ is surjective, i.e. $$z\notin \phi(\Z/2^i\Z\times G)=2\Z/2^i\Z\times G.$$
   Moreover, in this situation, the condition $\phi(z)=\phi(y)-2z$ is always fulfilled.

   Notice that we obtain the skew brace $J_{2^{i+1}d}(G)$ choosing $y=(0,0)$ and $z=(1,0)$
   and the skew brace $H_{2^{i+1}d}(G)$ choosing $y=(2^{i-1},0)$ and $z=(1,0)$.
   It remains to prove that $J_{2^{i+1}d}(G)$ and $H_{2^{i+1}d}(G)$ are not isomorphic 
   and that any other choice produces a skew brace isomorphic to either one of these two.
   
   The first claim is easy to prove since
   if $J_{2^{i+1}d}(G)$ and $H_{2^{i+1}d}(G)$ were isomorphic,
   by Theorem ~\ref{thm: onevertex isomorphic condition},
   there would be $\sigma\in\Aut(\Z/2^i\Z\times G)$ such that $\sigma(0,0)=\sigma(2^{i-1},0)$,
   which is not possible.

   Fix any skew brace $A_1$ defined using $y=(0,0)$ and $z\notin 2\Z/2^i\Z\times G$.
   Then $\sigma=\id_{\Z/2^i\Z\times G}$ satisfies the conditions
   of Theorem ~\ref{thm: onevertex isomorphic condition} for $A=J_{2^{i+1}d}(G)$ and $A_1$,
   thus $A_1$ is isomorphic to $J_{2^{i+1}d}(G)$.

   Let now $A_1$ be any skew brace 
   defined using $y=(2^{i-1},0)$ and $z\notin 2\Z/2^i\Z\times G$.
   Then $\sigma=\id_{\Z/2^i\Z\times G}$ satisfies the conditions
   of Theorem ~\ref{thm: onevertex isomorphic condition} for $A=H_{2^{i+1}d}(G)$ and $A_1$,
   thus $A_1$ is isomorphic to $H_{2^{i+1}d}(G)$.
\end{proof}

\begin{rem}
Let $G$ and $G'$ be abelian groups of odd order $d$, then $G$ is isomorphic to $G'$ if and only if $J_{2^{i+1}d}(G)$ and $J_{2^{i+1}d}(G')$ are isomorphic skew braces. 
Moreover, $G$ is isomorphic to $G'$ if and only if $H_{2^{i+1}d}(G)$ and $H_{2^{i+1}d}(G')$ are isomorphic skew braces. 
\end{rem}

\begin{thm}
\label{thm: onevertex:general}
A complete list of skew braces of size $2^md$, for $\gcd(2,d)=1$,
with one-vertex common divisor graph is as in Table~\ref{tab:onevertex}, where $G$ is any abelian group of order $d$
and $\A(d)$ is the number of isomorphism classes of 
abelian groups of order $d$.
\begin{table}[H]
\caption{The classification of skew braces
with a one-vertex $\lambda$-graph.}
\label{tab:onevertex}
\begin{center}
\begin{tabular}{|c|c|c|c|} 
 \hline
 $m$ & Isom. classes & Reference & \# Isom. classes\\ 
 \hline
 0 &  $-$ & Lemma~\ref{1vertex:leftnilp:lem} & 0\\ 
 \hline
 1 & $D_{2d}(G)$ & Theorem~\ref{1vertex:2odd:thm} & $\A(d)$ \\ 
 \hline
 \multirow{2}{*}{$2$}& $J_{4d}(G)$ & Proposition~\ref{pro:onevertex 2^id} & \multirow{2}{*}{$2\A(d)$}\\
    & $H_{4d}(G)$ & Proposition~\ref{pro:onevertex 2^id} & \\ 
\hline
 \multirow{3}{*}{3}& $K_{8d}(G)$ & Proposition~\ref{pro:onevertex 8d} & \multirow{3}{*}{$3\A(d)$}\\
    & $J_{8d}(G)$ & Proposition~\ref{pro:onevertex 2^id} & \\
    & $H_{8d}(G)$ & Proposition~\ref{pro:onevertex 2^id} & \\
\hline 
  \multirow{2}{*}{$\geq 4$}& $J_{2^md}(G)$ & Proposition~\ref{pro:onevertex 2^id} & \multirow{2}{*}{$2\A(d)$}\\
    & $H_{2^md}(G)$ & Proposition~\ref{pro:onevertex 2^id} & \\ 
 \hline
\end{tabular}
\end{center}
\end{table}
\end{thm}

\section{Skew braces with one-vertex \texorpdfstring{$\theta$}{theta}-graph}

\begin{lem}
\label{lem 1vertex theta is abelian}
Let $A$ be a finite skew brace. If $\Theta(A)$ has exactly one vertex, 
then $A$ is of abelian type.
\end{lem}

\begin{proof}
Suppose $(A,+)$ is non-abelian.
Let $\Theta(x)$ be the non trivial $\theta$-orbit.
Then $Z(A,+)\cap \Theta(x)=\emptyset$, otherwise,
being a characteristic subgroup,
$Z(A,+)$ would contain the whole orbit $\Theta(x)$, so $Z(A,+)=\Fix(A)$.
Thus $\Fix_\theta(A)=Z(A,+)$.
We claim that the additive order of $x$ is a prime power.
If not every Sylow $p$-subgroup of $(A,+)$ would be contained in the center,
which is not possible since $(A,+)$ is non-abelian.
Thus there exists a prime $p$ such that the additive order
of $x$ is a power of $p$.
Moreover for every prime $q\neq p$
every Sylow $q$-subgroup of $(A,+)$ is contained in the center.
Denote by $\mathcal{O}_{p'}$ a Hall $p'$-subgroup of $(A,+)$ and by $P$ a Sylow $p$-subgroup of $(A,+)$.
Then, by the Schur–Zassenhaus theorem and the fact that $\mathcal{O}_{p'}\subseteq Z(A,+)$, we get that
\[
(A,+)\cong \mathcal{O}_{p'}\times P.
\]
Thus $P$ is the unique Sylow $p$-subgroup of $(A,+)$ and it contains $\Theta(x)$. Moreover, setting $|P|=p^\alpha$, we have that $\Fix_\theta(A)=Z(A,+)$ has index $p^t$ for some $1<t\leq \alpha$.
Since $(A,+)$ is nonabelian, then $P$ must be non-abelian 
because $\mathcal{O}_{p'}$ is in the center.
But if $\mathcal{O}_{p'}$ is non-trivial, there must be elements 
of prime order $q\neq p$ in $\mathcal{O}_{p'}$, which implies the existence of 
a non-central element $y$ of order $p^s q$, for some $s\in \Z_{>1}$.
Therefore $y\in \Theta(x)$, which is a contradiction with the fact that the elements 
of $\Theta(x)$ have order a power of $p$.
Hence $A=P$ and so $|A|=p^\alpha$, $|Z(A,+)|=p^{\alpha-t}$ and 
$|\Theta(x)|=p^\alpha-p^{\alpha-t}=p^{\alpha-t}(p^t-1)$.
Moreover $|\Theta(x)|$ has to divide $|(A.+)\rtimes_\lambda(A,\circ)|$, i.e. 
$p^{\alpha-t}(p^t-1)$ divides $p^{2\alpha}$, which means that $p^t-1=1$.
Hence $Z(A,+)$ has index $p^t=2$, but this is not possible.
\end{proof}

\begin{thm}\label{thm:onevertextheta}
    Let $A$ be a finite skew brace of abelian type such that $\Theta(A)$ has only one vertex. Then $A$ is isomorphic to one of the following skew braces.
    \begin{itemize}
        \item The skew brace on $\Z/4\Z$, with multiplication given by $x\circ y=x+y+2xy$.
        \item The skew brace on $\Z/2\Z\times\Z/2\Z$, with multiplication 
        \[
        (x_1,y_1)\circ(x_2,y_2)=(x_1+x_2+y_1y_2, y_1+y_2).
        \]
        \item The skew brace on $\Z/2\Z\times\Z/4\Z$, with multiplication 
        \[
        (x_1,y_1)\circ (x_2,y_2)=\left(x_1+x_2+y_2\sum\limits_{i=1}^{y_1-1}i,y_1+y_2+2y_1y_2
        \right).
        \]
    \end{itemize}
\end{thm}
\begin{proof}
By Lemma~\ref{lem 1vertex theta is abelian}, $A$ is of abelian type, so $\Lambda(A)=\Theta(A)$ and we can use the classification in Theorem~\ref{thm: one vertex lambda of abelian type} to complete the proof.
\end{proof}

\section*{Acknowledgements}
The first author is partially supported by Vrije Universiteit Brussel via the projects OZR4014 and
OZR3762, and by Fonds Wetenschappelijk Onderzoek - Vlaanderen, via a PhD Fellowship fundamental
research, grant 11PIO24N and via the Senior Research Project G004124N.
The authors thank Leandro Vendramin for insightful discussions. Moreover, the authors thank the anonymous referee for their detailed reading of the manuscript and the helpful comments that were provided.
\bibliographystyle{abbrv}
\bibliography{refs}

\end{document}